\documentclass[11pt]{basicstyle}

\citesort
\usepackage{booktabs}

\usepackage{tikz}
\usetikzlibrary{matrix,arrows}

\DeclareMathOperator{\GL}{GL}
\DeclareMathOperator{\U}{U}
\DeclareMathOperator{\SU}{SU}
\DeclareMathOperator{\Aut}{Aut}
\DeclareMathOperator{\Inn}{Inn}
\DeclareMathOperator{\Sym}{Sym}
\DeclareMathOperator{\Alt}{Alt}
\DeclareMathOperator{\Norm}{N}

\DeclareMathOperator{\ST}{ST}

\newcommand{\C}{\mathbb{C}}
\newcommand{\D}{\mathbb{D}}
\newcommand{\R}{\mathbb{R}}
\newcommand{\HH}{\mathbb{H}}
\newcommand{\sS}{\mathbb{S}}
\newcommand{\id}{\mathbf{1}}

\newcommand{\Tr}{\top}
\newcommand{\soi}[1]{[\![#1]\!]}

\renewcommand{\(}{\langle\mskip2mu\relax}
\renewcommand{\)}{\mskip2mu \rangle}

\newcommand{\dotcup}{\mathrel{\dot\cup}}
\let\normal=\trianglelefteq

\theoremstyle{mythm}
\newtheorem{thm}{Theorem}[section]
\newtheorem{cor}[thm]{Corollary}
\newtheorem{lemma}[thm]{Lemma}
\theoremstyle{mydefn}
\newtheorem{defn}[thm]{Definition}
\newtheorem{rem}[thm]{Remark}
\newtheorem{rems}[thm]{Remarks}
\numberwithin{equation}{section}
\newtheorem{example}[thm]{Example}

\makeatletter
\long\def\@makecaption#1#2{\vspace*{8\p@}{%
        \setbox0=\vbox{\footnotesize\baselineskip=9\p@
        \color{mybrown}\rm #1.\ #2}%
        \setbox0=\vbox{\unvbox0 \setbox1=\lastbox
        \setbox1=\hbox to \textwidth
            {\hfill\unhbox1\hfill}\box1}\box0}\par
        \vspace*{3\p@}}
\makeatother

\begin{document}
\runningtitle{Rank two quaternionic reflection groups}
\title{Systems of imprimitivity for\\[4pt]
rank two quaternionic reflection groups}

\author{D. E. Taylor}
\address{School of Mathematics and Statistics, 
The University of Sydney
\email{Donald.Taylor@sydney.edu.au}}

\authorheadline{D. E. Taylor}

\version{27 July 2026}

\begin{abstract}
We revise the enumeration of the imprimitive rank two quaternionic reflection 
groups, adding missing groups and establishing isomorphisms between
groups in the published tables. The isomorphisms are obtained as a
consequence of the determination of the reflection groups with more than
one system of imprimitivity. We find that there are primitive complex
reflection groups which have infinitely many systems of imprimitivity
when represented as quaternionic reflection groups.
\end{abstract}

\def\classtag{2020}%

\maketitle

\section{Introduction}
In 1978 the finite quaternionic reflection groups were classified by
Arjeh Cohen~\cite{cohen:1978,cohen:1980}. When the group is irreducible and 
its complexification is primitive the results are derived from the 
classifications of Blichfeldt \cite{blichfeldt:1917}, Huffman and Wales 
\cite{huffman-wales:1977,wales:1978}. 

In recent years, the connection with the generalised McKay correspondence
and symplectic resolutions of quotient singularities has sparked renewed
interest in these groups.  See, for example, 
\cite{bellamy-schedler:2016,bellamy-schmitt-thiel:2023} and the references 
given there.

However, as pointed out by Yamagishi \cite{yamagishi:2018} and others, there
are omissions in the published tables and some groups appear more than once. 
The main results of the present paper are Theorems \ref{thm:bin_dihedral},
\ref{thm:binpol_11} and \ref{thm:binpol_theta}: the
determination of the rank two quaternionic reflection groups with more than 
one system of imprimitivity. In many cases the existence of more than one 
system corresponds to conjugacy between groups in the tables. In addition
there are irreducible complex reflection groups (both primitive and
imprimitive) which have infinitely many systems of imprimitivity when
represented as quaternionic reflection groups. (Remarks \ref{rem:inf2}
and \ref{rem:inf1}.)

Table \ref{tbl:5}, at the end of this paper, is the revised list of
imprimitive irreducible proper rank two quaternionic reflection groups.
See \cite{waldron:2025} for another approach.

\section{Notation}\label{sec:notation}
Let $\HH$ denote the division algebra of real quaternions. Its elements
are expressions $q = a_0 + a_1i + a_2 j + a_3 k$, where
$a_0$, $a_1$, $a_2$, $a_3\in\R$, and $i^2 = j^2 = k^2 = ijk = -1$.  The
conjugate of $q$ is $\overline q = a_0 - a_1i - a_2 j - a_3 k$, its norm 
is $\Norm(q) = q\bar q$ and its trace is $q + \bar q$. The norm is
multiplicative and therefore the 3-sphere $\sS^3$ of quaternions of norm 1
is a group. For more details see \cite{conway-smith:2003,voight:2021,
lehrer-taylor:2009}.

Let $V$ be a left $\HH$-space of dimension $n$ with hermitian inner product
$(-\mid-)$. Throughout this paper the inner product of row vectors 
$u = (u_1,\dots,u_n)$ and $v = (v_1,\dots,v_n)$ in $V$ is
$(u\mid v) = \sum_{h=1}^n u_h\overline v_h$. An $n\times n$ matrix of 
quaternions acts on a row vector by multiplication on the right.
The \emph{quaternionic unitary group} $\U_n(\HH)$ is the group of
$n\times n$ matrices acting on $V$ that preserve the inner product; 
i.e., matrices $A$ such that $A\overline{A}^\Tr = I$, where 
$\overline{A}^\Tr$ is the transposed conjugate of $A$.

A group $G\subseteq \U_n(\HH)$ acting on $V$ is \emph{imprimitive} if for 
some $m > 1$ the space $V$ is a direct sum $V_1\oplus \dots\oplus V_m$ of 
non-zero subspaces $V_h$ such that the action of $G$ permutes the $V_h$ 
among themselves. The set $\{V_1, \dots, V_m\}$ is a \emph{system of 
imprimitivity} for~$G$.

\begin{defn}
When the rank of $G$ is two and $G$ is imprimitive we denote the system of
imprimitivity $\{\(u\),\(v\)\}$ by $\soi{u,v}$. Thus $\soi{e_1,e_2}$ is the 
system of imprimitivity where $e_1 = (1,0)$ and $e_2 = (0,1)$ are the 
standard basis vectors. For reflection groups the vectors $u$ and $v$ are
orthogonal and so $\soi{u,v}$ is uniquely determined by $u$.
\end{defn}

A \emph{reflection} in $V$ is an element $s\in\U_n(\HH)$ of finite order
whose space $W$ of fixed points has dimension $n-1$. A \emph{root} of 
the reflection $s$ is a vector that spans the orthogonal complement of $W$.
If $a$ is a root of $s$, then $as = \xi a$ for some $\xi\in\sS^3$. It follows 
that
\[
  vs = v - \frac{(v\mid a)}{(a\mid a)}(1-\xi)a.
\]
Conversely, given a non-zero vector $a\in V$ and $\xi\in\sS^3$ let $s_{a,\xi}$ 
denote the linear transformation defined by the above formula. For $\xi\ne 1$, 
it is a reflection. A \emph{quaternionic reflection group} is a subgroup $G$ of 
$\U_n(\HH)$ that is generated by quaternionic reflections. A familiar argument
(for example, \cite[Theorem 1.27]{lehrer-taylor:2009}) shows that $G$ decomposes 
as a direct product of irreducible reflection subgroups. Furthermore, if $G$ is
irreducible and imprimitive, the subspaces in a system of imprimitivity are 
pairwise orthogonal of dimension 1; in this case the group can be 
represented by monomial matrices.

For all $g\in\U_n(\HH)$ we have $g^{-1}s_{a,\xi}g = s_{ag,\xi}$. For all
$\lambda\in \HH^\times$ we have $s_{\lambda a,\xi} = s_{a,\lambda^{-1}\xi\lambda}$
and for all $\lambda,\eta\in \sS^3\setminus\{1\}$ we have 
$s_{a,\xi}s_{a,\eta} = s_{a,\xi\eta}$.

For every $\theta\in\HH$ such that $\theta^2 = -1$, the subspace spanned by 1 
and $\theta$ can be identified with the field of complex numbers. In particular, 
$\C$ will denote the span of 1 and $i$ in $\HH$. Under this identification, a 
subgroup $G$ of $\U_n(\C)$ --- unitary matrices in the usual sense ---
becomes a subgroup $G^\sharp$ of $\U_n(\HH)$. Following \cite{cohen:1980} such a 
group is said to be of \emph{complex type} and a group not of complex type is
said to be a \emph{proper quaternionic group}.

The left $\HH$-space $V$ with standard basis $e_1$, $e_2$, \dots,~$e_n$
may be regarded as a $\C$-space $V_\C$ with basis $e_1$, $e_2$, \dots,~$e_n$, 
$je_1$, $je_2$, \dots,~$je_n$. If $A$ is an $n\times n$ quaternionic matrix,
then $A = A_1 + A_2j$, where $A_1$ and $A_2$ are $n\times n$ complex
matrices. The action of $A$ on $V_\C$ is given by the matrix
$A^\circ = \begin{pmatrix}A_1&A_2\\[2pt]-\overline A_2&\overline A_1\end{pmatrix}$.
The \emph{complexification} of a group $G$ of quaternionic matrices is
$G^\circ = \{\,A^\circ \mid A\in G\,\}$.

As indicated in \cite[\S4.1]{bellamy-schedler:2016} an irreducible quaternionic
reflection group $G$ is of complex type $H^\sharp$ for some irreducible
complex reflection group $H$ if and only if $G^\circ$ is reducible.
This is implicit in \cite{cohen:1980} and a short proof can be found in
\cite[\S1.2.2]{schmitt:2023}.

The irreducible primitive and imprimitive complex reflection groups $G_n$ and
$G(m,p,n)$ of Shephard and Todd \cite{shephard-todd:1954,lehrer-taylor:2009} 
will be denoted by $\ST(n)$ and $\ST(m,p,n)$ when considered as 
quaternionic reflection groups of complex type.

Depending on the context, $\id$ denotes the identity automorphism or the 
trivial group.

\section{The finite subgroups of $\sS^3$}\label{sec:binpol}
The finite subgroups of the multiplicative group of quaternions are the
cyclic groups and binary polyhedral groups \cite{lehrer-taylor:2009,voight:2021}. 
The following elements of $\sS^3$ are used in their definitions.
\begin{gather*}
  \zeta_m = \cos\,(2\pi/m) + i\sin\,(2\pi/m),\quad
  \varpi = \tfrac12(-1 + i + j + k),\\
  \gamma = \tfrac1{\sqrt2}(1 + i),\quad
  \sigma = \tfrac12(\tau^{-1} + i + \tau j), \text{ where } 
  \tau = \tfrac12(1 +\sqrt5).
\end{gather*}

\noindent
Every finite subgroup of $\HH^\times$ is conjugate in $\sS^3$ to one of
the following.
\begin{enumerate}[\enspace(i)]
\item The cyclic group $\mathcal C_m = \(\zeta_m\)$ of order $m \ge 1$.
\item The binary dihedral group $\mathcal D_m = \(\zeta_{2m}, j\)$
 of order $4m$, ($m \ge 1$).
\item The binary tetrahedral group $\mathcal T = \(i,j,\varpi\)$
 of order 24.
\item The binary octahedral group $\mathcal O = \(i,j,\varpi,\gamma\)$
 of order 48.
\item The binary icosahedral group $\mathcal I = \(i,\sigma\)$,
 of order 120.
\end{enumerate}
(For the convenience of later notation the cyclic group of order 4 occurs 
twice: as $\mathcal C_4 = \(i\)$ and $\mathcal D_1 = \(j\)$.)

\smallskip\noindent
The binary polyhedral groups modulo their central subgroup 
$\mathcal C_2 = \{\pm1\}$ are the following quotients.
\begin{enumerate}[\enspace(i)]
\item $\mathcal D_m/\mathcal C_2\simeq\D_m = \( u, v \mid u^m = v^2 = (uv)^2 = 1\)$,
the dihedral group of order $2m$.
\item $\mathcal T/\mathcal C_2\simeq \Alt(4)$, the alternating group of order 12.
\item $\mathcal O/\mathcal C_2\simeq \Sym(4)$, the symmetric group of order 24.
\item $\mathcal I/\mathcal C_2\simeq \Alt(5)$, the alternating group of order 60.
\end{enumerate}
(This notation includes $\D_1\simeq \mathcal C_2$ and
$\D_2\simeq \mathcal C_2\times\mathcal C_2$.)

\section{The irreducible imprimitive groups of rank two}\label{sec:imprim2}
We begin by reviewing the construction of the groups $G(K,H,\varphi)$ defined 
in \S2 of \cite{cohen:1980}, where $K$ is a finite group of quaternions, $H$ is a
normal subgroup of $K$ and $\varphi$ is an automorphism of $K/H$ such that 
$\varphi^2 = \id$.

Let $\mu : K \to K/H$ be the natural homomorphism and let $K\times_\varphi K$ 
be the equaliser of $\varphi\mu$ and $\mu$. That is, 
$K\times_\varphi K = \{\,(\xi,\eta)\in K\times K\mid\varphi(H\xi) = H\eta\,\}$. 
For $(\xi,\eta)\in K\times_\varphi K$ the map $\sigma(\xi,\eta) = (\eta,\xi)$ is an 
automorphism and the semidirect product $(K\times_\varphi K)\(\sigma\)$ has a faithful 
representation as a subgroup $G(K,H,\varphi)$ of $\U_2(\HH)$ where $(\xi,\eta)$ 
corresponds to $\begin{pmatrix}\xi&0\\0&\eta\end{pmatrix}$ and $\sigma$ corresponds 
to $\begin{pmatrix}0&1\\1&0\end{pmatrix}$. The order of $G(K,H,\varphi)$ is
$2|K|\;|H|$.

\begin{defn}\label{defn:L}
$L_\varphi = \{\,\xi\in K \mid\varphi(H\xi) = H\xi^{-1}\,\}$.
\end{defn}

\begin{lemma}\label{lemma:gen}
$G(K,H,\varphi)$ is a quaternionic reflection group if and only if $L_\varphi$ 
generates~$K$. 
\end{lemma}

\begin{proof}
This follows from the observation that the reflections in 
$G(K,H,\varphi)$ are the matrices
\[
  \begin{pmatrix}1&0\\0&\xi\end{pmatrix},\enspace
  \begin{pmatrix}\xi&0\\0&1\end{pmatrix}\enspace
    1\ne\xi\in H\quad\text{and}\quad
  \begin{pmatrix}0&\xi\\\xi^{-1}&0\end{pmatrix}\enspace\text{for}\ \xi\in L_\varphi.
  \qedhere
\]
\end{proof}

\begin{example}\label{ex:1}
We have $H\subseteq L_\varphi$. If the commutator of $\xi,\eta\in L_\varphi$ is 
in $H$, then $\xi\eta\in L_\varphi$. Therefore, if $K/H$ is abelian, $L_\varphi$ 
is a subgroup of $K$ and in this case $L_\varphi$ generates $K$ if and only 
if~$L_\varphi = K$.

Consequently, if $H$ is a subgroup of index 1 or 2 in a binary polyhedral 
group $K$, it is always the case that $G(K,H,\id)$ is a quaternionic reflection 
group.
\end{example}

\begin{thm}[{Cohen \cite[Theorem (2.2)]{cohen:1980}}]
Suppose that $G$ is an irreducible imprimitive quaternionic reflection group 
acting on $V = \HH^2$ and $\{V_1,V_2\}$ is a system of imprimitivity for $G$.
Then $G$ is a group $G(K,H,\varphi)$ for suitable $K$, $H$ and $\varphi$.  
\end{thm}

Cohen shows that there is a reflection $s\in G$ 
interchanging $V_1$ and $V_2$ and that $e_1\in V_1$ and $e_2\in V_2$ may be chosen
so that the matrix of $s$ with respect to the basis $e_1$, $e_2$ is 
$\begin{pmatrix}0&1\\1&0\end{pmatrix}$. In fact $K$ is the set of $\xi$ such that 
$\begin{pmatrix}\xi&0\\0&\eta\end{pmatrix}\in G$ for some $\eta$ and $H$ is 
the subset of those $\xi\in K$ such that $\eta = 1$. The map $\xi\mapsto \eta$
extends to a well-defined automorphism of $K/H$ of order 1 or~2. Neither $K$, $H$, 
nor $\varphi$ is uniquely determined by~$G$. Indeed, the labelling depends on the 
choice of a system of imprimitivity.

\begin{defn}\label{defn:rho}
Suppose that $\xi\in\sS^3$ normalises both $K$ and $H$. Define $\rho(\xi)$ to
be the automorphism of $K/H$ that sends $H\eta$ to $H\xi\eta\xi^{-1}$.
\end{defn}

\begin{lemma}[{\cite[Lemma 2.4]{cohen:1980}}]\label{lemma:conj}
Suppose $\varphi\in\Aut(K/H)$ and $\varphi^2 = \id$.
\begin{enumerate}[\enspace(i)]
\item If $\xi\in L_\varphi$ and $h = \begin{pmatrix}1&0\\0&\xi\end{pmatrix}$, 
then $hG(K,H,\varphi)h^{-1} = G(K,H,\rho(\xi)\varphi)$.
\item If $\xi\in\sS^3$ normalises $K$ and $H$, then $hG(K,H,\varphi)h^{-1}
  = G(K,H,\rho(\xi)\varphi\rho(\xi)^{-1})$, where 
  $h = \begin{pmatrix}\xi&0\\0&\xi\end{pmatrix}$.\qed
\end{enumerate}
\end{lemma}

\begin{cor}\label{cor:id}
If $\xi\in L_{\rho(\xi)}$, then $G(K,H,\rho(\xi))$ is conjugate to $G(K,H,\id)$.
\end{cor}

\subsection{The cyclic family.}
\begin{thm}[{\cite[(2.5)]{cohen:1980}}]\label{thm:cyclic}
Suppose that $G(\mathcal C_m,H,\varphi)$ is an imprimitive quaternionic 
reflection group. Then $H = \mathcal C_\ell$ for a divisor 
$\ell$ of $m$ and $\varphi(H\xi) = H\xi^{-1}$. Furthermore, 
$G(\mathcal C_m,\mathcal C_\ell,\varphi)\simeq \ST(m,m/\ell,2)$.
\end{thm}

\begin{proof}
Certainly $H = \mathcal C_\ell$ where $m = n\ell$. Since $\mathcal C_m$ is 
abelian, $L_\varphi = \mathcal C_m$ and the only choice for $\varphi$ is 
$\varphi(H\xi) = H\xi^{-1}$ for all $\xi\in\mathcal C_m$. Consequently, 
the matrices
\[
  g_1 = \begin{pmatrix}0&1\\1&0\end{pmatrix},\quad
  g_2 = \begin{pmatrix}\zeta_m&0\\0&\zeta_m^{-1}\end{pmatrix},\quad
  g_3 = \begin{pmatrix}\zeta_m^n&0\\0&1\end{pmatrix}
\]
generate $G(\mathcal C_m,\mathcal C_\ell,\varphi)$.  Then $g_1$, $g_1g_2$ 
and $g_3$ are generators of $\ST(m,n,2)$ described in 
\cite[Ch.\ 2, \S7]{lehrer-taylor:2009} and thus
$G(\mathcal C_m,\mathcal C_\ell,\varphi)\simeq\ST(m,n,2)$.
\end{proof}

The Shephard and Todd group $\ST(2m,m,2)$ for $m > 2$, regarded as a complex 
reflection group, has only one system of imprimitivity whereas, as can be
seen in \S\ref{sec:imprim}, Table \ref{tbl:3}, when represented as a quaternionic 
reflection group of complex type, it has infinitely many. 

\subsection{The binary dihedral family $G(\mathcal D_m,H,\varphi)$.}
As an abstract group the binary dihedral group $\mathcal D_m$ has a presentation
\begin{equation}\label{eq:K}
  K = \( x, y\mid x^m = y^2 = (xy)^2 \).
\end{equation}
It follows that $y^{-1}xy = x^{-1}$, hence $x^m = y^{-1}y^2y = y^{-1}x^my = x^{-m}$ 
and consequently $z^2 = 1$ where $z = x^m = y^2$. Therefore $(x^hy)^2 = z$ for 
all $h$ and the elements of $K$ are $x^r$ and $x^ry$ for $0\le r < 2m$. There is
an isomorphism between $K$ and $\mathcal D_m$ mapping $x$ to $\zeta_{2m}$, $y$ 
to $j$ and $z$ to $-1$.

\subsubsection*{Case 1: $H$ cyclic.}
Suppose that $H$ is a cyclic normal subgroup of $K$. The subgroup $\(x\)$ is 
normal in $K$ and for $m\ne 2$ it is characteristic. For $m > 2$,
$H\subseteq \(x\)$ and so $H = \(x^n\)\simeq\mathcal C_\ell$ where $2m = n\ell$.
If $m = 2$, $K$ is the quaternion group of order 8 and conjugacy in $\mathcal T$ 
shows that we may assume $H\subseteq \(x\)$.

If $\ell$ is odd, then $n$ is even and we claim that $K/H\simeq\mathcal D_{m_0}$,
where $m_0 = n/2$. Certainly $(Hy)^{-1}(Hx)(Hy) = Hx^{-1}$ and since $\ell$ is odd, 
$y^2\notin H$. We need to show that $(Hx)^{m_0} = (Hy)^2$. Indeed,
$y^{-2}x^{m_0} = x^{-m}x^{m_0} = x^{n(\ell-1)/2}\in H$, as required.
(Note that if $m$ is odd, $K = \mathcal D_m$ and $H = \mathcal C_m$, then 
$K/H\simeq\mathcal D_1\simeq\mathcal C_4$.) 

If $\ell$ is even, then $z\in H$ and a similar calculation shows that 
$K/H\simeq\D_n$, the dihedral group of order $2n$. (Recall that 
according to the conventions of \S\ref{sec:binpol},
$\D_1\simeq\mathcal C_2$ and $\D_2\simeq\mathcal C_2\times\mathcal C_2$.)

For $r \ge 1$, $s\ge 0$, the map $K\to K : x^ay^b\mapsto x^{ar+bs}y^b$ for 
$0\le a < 2m$, $b \in\{0,1\}$ preserves $H$ and induces an automorphism 
$\varphi_{r,s}$ of $K/H$, which has order 1 or 2 if and only if
\[
  \gcd(r,n) = 1,\enspace n\mid r^2-1\text{ and }n\mid s(r+1). 
\]
In fact the maps $\varphi_{r,s}$ account for all the automorphisms of
$K/H$ of order 1 or 2.

By definition, $\xi\in K$ belongs to $L_{\varphi_{r,s}}$ if and only if 
$\varphi_{r,s}(H\xi) = H\xi^{-1}$ and so $L_{\varphi_{r,s}} = L'\dotcup L''$, 
where
\begin{equation}\label{eq:LL}
  L' = \bigcup_{\eta\in H}\,\{\,\eta x^h\mid x^{h(r+1)}\in H\,\}
  \enspace\text{and}\enspace
  L'' = \bigcup_{\eta\in H}\,\{\,\eta x^h y\mid x^{h(r-1)+s+m}\in H\,\}.
\end{equation}

\begin{lemma}\label{lemma:L}
We have $H\subseteq L' = \(x^\kappa\)\normal K$, where $\kappa = n/\gcd(n,r+1)$.
\end{lemma}

\begin{proof}
The elements of $L'$ commute and therefore $L'$ is a subgroup of $\(x\)$.
Thus $L' = H\(x^\kappa\)$, where $\kappa$ is the smallest divisor of $n$
such that $x^{\kappa(r+1)}\in H$. Write $\gcd(n,r+1) = an + b(r+1)$ for 
some $a,b$. Suppose that $h\mid n$ and $x^{h(r+1)}\in H$. Then
$h\gcd(n,r+1) = ahn + bh(r+1)$, whence $h\ge n/\gcd(n,r+1)$ since
$n\mid h(r+1)$. Therefore $\kappa = n/\gcd(n,r+1)$ and so $L' = H\(x^\kappa\) = 
\(x^n\)\(x^\kappa\) = \(x^{\gcd(n,\kappa)}\) = \(x^\kappa\)$.
\end{proof}

\begin{lemma}
Given $K = \(x,y\)$ as in \eqref{eq:K} and a divisor $n$ of $2m$, let
$H = \(x^n\)$. If $G(K,H,\varphi_{r,s})$ is a quaternionic reflection 
group, then up to conjugacy we may suppose that $s = m$ and either 
$r = 1$ or $2r \le n$.
\end{lemma}

\begin{proof}
If $n = 1$, the only automorphism of $K/H$ is the identity, which can be 
written as $\varphi_{1,m}$ since $x^m\in H$. Thus we may suppose that $1\le r < n$.
By Lemma \ref{lemma:gen}, $L_{\varphi_{r,s}}$ generates $K$ and it follows from 
Lemma~\ref{lemma:L} that $L''\ne\emptyset$. For $x^hy\in L''$ a straightforward 
calculation shows that $\rho(x^hy)\varphi_{r,s} = \varphi_{2m-r,2m+2h-s}$.  Since 
$x^{2m-r}\equiv x^{n-r}\pmod H$, the map $\rho(x^hy)\varphi_{r,s}$ is the
automorphism $\varphi_{n-r,n+2h-s}$ of $K/H$. It follows from 
Lemma~\ref{lemma:conj}\,(i) that up to conjugacy we may suppose $1\le r\le n/2$.

To carry out the next step, embed $K$ in 
\[
  \widetilde K = \( w, y\mid w^{2m} = y^2 = (wy)^2 \) \simeq \mathcal D_{2m}
\]
by defining $x$ to be $w^2$. Then $w$ normalises both $K$ and $H$ and by
direct calculation we have 
$\rho(w^hy)\varphi_{r,s}\rho(w^hy)^{-1} = \varphi_{r,2m-h(r-1)-s}$,
If $x^hy\in L''$, then $x^{h(r-1)+s+m}\in H$ whence 
$\varphi_{r,2m-h(r-1)-s} = \varphi_{r,m}$.  It follows from 
Lemma~\ref{lemma:conj}\,(ii) that $G(K,H,\varphi_{r,s})$ is conjugate to 
$G(K,H,\varphi_{r,m})$ in~$\U_2(\HH)$.
\end{proof}

\begin{lemma}
If $L''$ is the subset of $K$ associated with $G(K,H,\varphi_{r,m})$
as in \eqref{eq:LL}, then $L'' = \(x^\nu\)y$ where $\nu = n/\gcd(n,r-1)$.
\end{lemma}

\begin{proof}
For $\eta\in H$ it follows from \eqref{eq:LL} with $s = m$ that $\eta x^hy\in L''$
if and only if $x^{h(r-1)}\in H$. Therefore $L'' = H\(x^\nu\)y$, where
$\nu$ is the smallest divisor of $n$ such that $x^{\nu(r-1)}\in H$. An
adaptation of the proof of Lemma~\ref{lemma:L} shows that $\nu = n/\gcd(n,r-1)$
and $L'' = H\(x^\nu\)y = \(x^n\)\(x^\nu\)y = \(x^{\gcd(n,\nu)}\)y = \(x^\nu\)y$.
\end{proof}

\begin{cor}
$L_{\varphi_{r,m}}$ generates $K$ if and only if $\gcd(\kappa,\nu) = 1$.
\end{cor}
\begin{proof}
The subgroup generated by $L_{\varphi_{r,m}}$ contains 
$L'L'' = \(x^\kappa\)\(x^\nu\)y = \(x^{\gcd(\kappa,\nu)}\)y$.
\end{proof}

\begin{defn}
Let $\psi_r$ denote the automorphism $\varphi_{r,m}$ of $K/H$. That is,
$\psi_r(Hx) = Hx^r$ and $\psi_r(Hy) = Hy^{-1}$.
\end{defn}
The results of this section for cyclic $H$ are summarised in the following theorem.

\begin{thm}\label{thm:A}
Suppose that $G = G(K,H,\varphi)$ where $K$ is a binary dihedral group of order
$4m$ and $H$ is a cyclic subgroup of order $\ell$ such that $2m = n\ell$.
Then $G$ is a quaternionic reflection group if and only if it is 
conjugate in $\U_2(\HH)$ to $G(\mathcal D_m,\mathcal C_\ell,\psi_r)$ with $r = 1$ or 
$1 < r \le n/2$ and $\gcd(r,n) = \gcd(\kappa,\nu) = 1$, where 
$\kappa = n/\gcd(n,r+1)$ and $\nu = n/\gcd(n,r-1)$.

As generators we may take
\begin{equation}\label{eq:std-gens}
  \begin{pmatrix}0&1\\1&0\end{pmatrix},\quad
  \begin{pmatrix}x^n&0\\0&1\end{pmatrix},\quad
  \begin{pmatrix}x&0\\0&x^r\end{pmatrix},\quad
  \begin{pmatrix}y&0\\0&y^{-1}\end{pmatrix}.
\end{equation}
\end{thm}

\begin{rems}\leavevmode
\begin{enumerate}[1)]
\item 
The groups $G(\mathcal D_m,\mathcal C_\ell,\psi_r)$ where $2m = n\ell$ 
with $n > 1$, and where $\ell$ does not divide $m$, are not listed in
Table I of \cite{cohen:1980}.
\item
It follows from the descriptions of $L'$ and $L''$ that 
$G(\mathcal D_m,\mathcal C_\ell,\psi_r)$ is generated by the reflections
\begin{equation}\label{eq:rk2gen}
  \begin{pmatrix}x^n&0\\0&1\end{pmatrix},\quad
  \begin{pmatrix}0&x^n\\x^{-n}&0\end{pmatrix},\quad
  \begin{pmatrix}0&x^\kappa\\x^{-\kappa}&0\end{pmatrix},\quad
  \begin{pmatrix}0&y\\y^{-1}&0\end{pmatrix},\quad
  \begin{pmatrix}0&x^\nu y\\x^\nu y^{-1}&0\end{pmatrix}
\end{equation}

The groups $G(\mathcal D_{2m-1},\mathcal C_{2m-1},\psi_1)$ and
$G(\mathcal D_{2m},\mathcal C_{4m},\psi_1)$ can be generated by
three reflections whereas $G(\mathcal D_m,\mathcal C_m,\psi_1)$ requires 
four.  Some groups such as $G(\mathcal D_{15},\mathcal C_2,\psi_4)$
and $G(\mathcal D_{21},\mathcal C_2,\psi_8)$ cannot be generated by fewer 
than five reflections.
\item
The first two rows of Table \ref{tbl:1} list the imprimitive groups 
$G(\mathcal D_m,H,\varphi)$. Not every group is proper: for all $m$
$G(\mathcal D_m,\id,\psi_1)$ is of complex type and isomorphic to 
$\ST(2m,m,2)$. Moreover there are pairs of groups which are conjugate
in $\U_2(\HH)$. See \S\ref{sec:bconj} for the details.
\end{enumerate}
\end{rems}

\begin{defn}\label{defn:stdDC}
If $2m = n\ell$, the group $G(\(\zeta_m,j\),\(\zeta_m^n\),\psi_r)$ where
$r$ satisfies the conditions of Theorem \ref{thm:A} is the 
\emph{standard copy} of $G(\mathcal D_m,\mathcal C_\ell,\psi_r)$.
\end{defn}

\subsubsection*{Case 2: $H$ non-cyclic.}
To complete the description of the groups $G(K,H,\varphi)$ where $K$ is
a binary dihedral group, suppose that $K$ is the group defined in 
\eqref{eq:K} where $m \ge 2$ and $H$ is a non-cyclic normal subgroup of $K$. 
Then $x^sy\in H$ for some $s$ and $x^{-2} = x^{-1}(x^sy)^{-1}x(x^sy)\in H$. 
If $m$ is odd, then $H = K$. But if 
$m$ is even and $H\ne K$, then $m > 2$ and $H$ is $\(x^2,y\)$ or $\(x^2,xy\)$.
The two possibilities for $H$ are conjugate in $\mathcal D_{2m}$ and 
isomorphic to $\mathcal D_{m/2}$. 

\begin{table}[ht]
\caption{The imprimitive rank two quaternionic reflection groups 
$G(\mathcal D_m,H,\varphi)$}
\label{tbl:1}
\medskip
\centering
\begin{tabular}{ccccc}
$K$&$H$&$K/H$&$\varphi$\\
\toprule
$\mathcal D_m$&$\mathcal C_\ell$&$\mathcal D_{m/\ell}$&$\psi_r$&$\ell$ odd, $\ell\mid m$\\
$\mathcal D_m$&$\mathcal C_\ell$&$\D_{2m/\ell}$&$\psi_r$&$\ell$ even, $\ell\mid 2m$\\
$\mathcal D_m$&$\mathcal D_m$&$\id$&$\id$&all $m$\\
$\mathcal D_m$&$\mathcal D_{m/2}$&$\mathcal C_2$&$\id$&$m$ even\\
\bottomrule
\end{tabular}

\medskip
$r = 1$ or $1 < r \le n/2$ and $\gcd(r,n) = \gcd(\kappa,\nu) = 1$.\\
$\kappa = n/\gcd(n,r+1)$, $\nu = n/\gcd(n,r-1)$, where $n = 2m/\ell$
\end{table}

\begin{thm}\label{thm:B}
The group $G(K,H,\varphi)$ --- where $H\normal K$ and $H$ and $K$ are non-cyclic 
binary dihedral groups --- is a quaternionic reflection group if and only if 
for some $m$ it is conjugate in $\U_2(\HH)$ to $G(\mathcal D_m,\mathcal D_m,\id)$ 
or $G(\mathcal D_{2m},\mathcal D_m,\id)$.
\end{thm}

\begin{proof}
The order of $K/H$ is 1 or 2 and $\id$ is its only
automorphism. Thus $G(\mathcal D_m,\mathcal D_m,\id)$ and 
$G(\mathcal D_{2m},\mathcal D_m,\id)$ are reflection groups for all 
$m\ge 2$. (See Example \ref{ex:1}).
\end{proof}

\begin{defn}\label{defn:stdDD}\relax\leavevmode
\begin{enumerate}[\enspace 1)]
\item
The group $G(\(\zeta_m,j\),\(\zeta_m,j\),\id)$ is the \emph{standard copy} of 
$G(\mathcal D_m,\mathcal D_m,\id)$.
\item
The group $G(\(\zeta_{2m},j\),\(\zeta_m,j\),\id)$ is the 
\emph{standard copy} of $G(\mathcal D_{2m},\mathcal D_m,\id)$.
\end{enumerate}
\end{defn}

Conjugacy between standard copies is dealt with in \S\ref{sec:conj}.

\subsection{The families of the binary polyhedral groups $\mathcal T$, 
$\mathcal O$ and $\mathcal I$.}

The cyclic group $\mathcal C_2 = \{\pm1\}$ is the centre of $\mathcal T$, 
$\mathcal O$ and $\mathcal I$ and there are unique chief series
\[
  \id\normal\mathcal C_2\normal\mathcal D_2\normal\mathcal T\normal\mathcal O
  \quad\text{and}\quad
  \id\normal\mathcal C_2\normal\mathcal I.
\]

\begin{lemma}
Suppose that $K$ is $\mathcal T$, $\mathcal O$ or $\mathcal I$. Then
$G = G(K,H,\id)$ is a quaternionic reflection group if and only if $G$ is
$G(\mathcal T,\mathcal T,\id)$,
$G(\mathcal O,\mathcal O,\id)$,
$G(\mathcal O,\mathcal T,\id)$,
$G(\mathcal O,\mathcal D_2,\id)$,
$G(\mathcal O,\mathcal C_2,\id)$,
$G(\mathcal I,\mathcal I,\id)$ or
$G(\mathcal I,\mathcal C_2,\id)$.
\end{lemma}

\begin{proof}
From Lemma~\ref{lemma:gen}, $G(K,H,\id)$ is a quaternionic reflection group 
if and only if $L = \{\,\xi\in K\mid\xi^2\in H\,\}$ generates $K$. 
If $|K/H|\le 2$, then $L = K$; if $H = \id$, then $L = \mathcal C_2$;
if $K = \mathcal T$ and $H = \mathcal D_2$, then $L = H$. This reduces the
proof to the cases 
$\mathcal T/\mathcal C_2\simeq \Alt(4)$, $\mathcal O/\mathcal D_2\simeq\Sym(3)$, 
$\mathcal O/\mathcal C_2\simeq\Sym(4)$ and $\mathcal I/\mathcal C_2\simeq\Alt(5)$. 
The result follows from the observation that only $\Alt(4)$ is not generated 
by its elements of order two.
\end{proof}

\begin{lemma}\label{lemma:1}
Suppose that $K$ is $\mathcal T$ or $\mathcal O$, $H\normal K$ and $\varphi$ is
an automorphism of $K/H$ of order 2 such that $G(K,H,\varphi)$ is a reflection group
not conjugate to $G(K,H,\id)$. Let $\delta = \frac{1}{\sqrt2}(i - j)$.
\begin{enumerate}[\enspace(i)]
\item If $K = \mathcal T$, then $H$ is $\id$, $\mathcal C_2$ or $\mathcal D_2$
and $G(\mathcal T,H,\varphi)$ is conjugate to $G(\mathcal T,H,\rho(\delta))$.
\item If $K = \mathcal O$, then $H = \id$ and $G(\mathcal O,\id,\varphi)$ is
conjugate to $G(\mathcal O,\id,\rho(\delta))$ or $G(\mathcal O,\id,\beta)$,
where $\beta$ is the central automorphism of order 2.
\end{enumerate}
\end{lemma}

\begin{proof}
We have
\begin{gather*}
  \Aut(\mathcal T)\simeq\Aut(\mathcal T/\mathcal C_2)\simeq
  \Aut(\mathcal O/\mathcal C_2)\simeq\Sym(4),\\
  \Aut(\mathcal O/\mathcal D_2)\simeq\Sym(3)
  \enspace\text{and}\enspace
  \Aut(\mathcal O)\simeq\(\beta\)\times\Sym(4),
\end{gather*} 
where $\beta$ is the central automorphism of $\mathcal O$ defined by 
$\beta(\xi) = \xi$ for $\xi\in\mathcal T$ and
$\beta(\xi) = -\xi$ for $\xi\in\mathcal O\setminus\mathcal T$.

Suppose at first that $\varphi = \rho(\gamma)$ for some $\gamma\in\mathcal O$
so that $L_\varphi = \{\,\xi\in K\mid \gamma\xi\gamma^{-1}\xi \in H\,\}$.
Then the order of $\gamma$ is 4 or 8. If the order is 8, then $K = \mathcal O$
and $\gamma^2\in\mathcal D_2$, hence $\mathcal D_2\subseteq H$. Thus 
$\gamma\in L_\varphi$ and from Corollary~\ref{cor:id}, $G(\mathcal O,H,\varphi)$ 
is conjugate to $G(\mathcal O,H,\id)$, contrary to assumption. Thus we may 
suppose that $\gamma$ is an element of order 4, hence 
$\gamma^2 = -1$ and $L_\varphi = \{\,\xi\in K\mid -(\gamma\xi)^2\in H\,\}$.

Suppose that $K = \mathcal T$. If $\gamma\in\mathcal D_2$, then 
$L_\varphi\subseteq\mathcal D_2$ whence $G(\mathcal T,H,\rho(\gamma))$ is not 
a reflection group. However, if $\gamma\in\mathcal O\setminus\mathcal T$,
then $L_\varphi$ generates $\mathcal T$ and so $G(\mathcal T,H,\rho(\gamma))$ 
is a reflection group for all $H$. There is a single conjugacy class of elements of
order 4 in $\mathcal O\setminus\mathcal T$, hence from Lemma~\ref{lemma:conj}\,(ii) 
we may take $\varphi = \rho(\delta)$. This completes the proof of~(i).

Suppose that $K = \mathcal O$. If $H\ne\id$, then $\gamma\in L_\varphi$.
From Corollary~\ref{cor:id}, $G(\mathcal O,H,\varphi)$ is conjugate to 
$G(\mathcal O,H,\id)$, contrary to assumption, hence $H = \id$. Since 
$k$ and $\delta$ represent the conjugacy classes of elements of order 4, 
$G(\mathcal O,\id,\varphi)$ is conjugate to $G(\mathcal O,\id,\rho(k))$ 
or $G(\mathcal O,\id,\rho(\delta))$.  But $k\in L_{\rho(\delta)}$ and 
from Lemma~\ref{lemma:conj}\,(i) we find that $G(\mathcal O,\id,\rho(k))$ 
is conjugate to $G(\mathcal O,\id,\rho(k\delta))$, which is conjugate to 
$G(\mathcal O,\id,\rho(\delta))$.

If the central automorphism $\beta$ acts non-trivially on $\mathcal O/H$, then
$H = \id$. In this case $L_\beta$ is the union of 
$\{1,-1\}$ and the conjugacy class of $\delta$. Thus $G(\mathcal O,\id,\beta)$ 
is a reflection group and from Lemma~\ref{lemma:conj}\,(i) it is conjugate to 
$G(\mathcal O,\id,\rho(\xi)\beta)$ for all conjugates $\xi$ of~$\delta$.
\end{proof}

\begin{table}[ht]
\caption{The reflection groups $G(K,H,\varphi)$ 
where $K$ is $\mathcal T$, $\mathcal O$ or $\mathcal I$}\label{tbl:2}
\smallskip
\centering
\begin{tabular}{lclr}
$G(K,H,\varphi)$&$K/H$&\qquad isomorphisms&order\\
\toprule
$G(\mathcal T,\mathcal T,\id)$&$\id$&&1\,152\\
$G(\mathcal T,\mathcal D_2,\rho(\delta))$&$\mathcal C_3$&&384\\
$G(\mathcal T,\mathcal C_2,\rho(\delta))$&$\Alt(4)$&
   ${}\simeq G(\mathcal O,\id,\beta)\simeq\mathcal C_4\boxdot_2\mathcal O$
   &96\\
$G(\mathcal T,\id,\rho(\delta))$&$\mathcal T$&
  ${}\simeq\ST(12)$ complex type&48\\
\midrule
$G(\mathcal O,\mathcal O,\id)$&$\id$&&4\,608\\
$G(\mathcal O,\mathcal T,\id)$&$\mathcal C_2$&&2\,304\\
$G(\mathcal O,\mathcal D_2,\id)$&$\Sym(3)$&&768\\
$G(\mathcal O,\mathcal C_2,\id)$&$\Sym(4)$&
  ${}\simeq\mathcal C_4\boxdot\mathcal O$&192\\
$G(\mathcal O,\id,\beta)$&$\mathcal O$&
  ${}\simeq G(\mathcal T,\mathcal C_2,\rho(\delta))\simeq\
  \mathcal C_4\boxdot_2\mathcal O$&96\\
$G(\mathcal O,\id,\rho(\delta))$&$\mathcal O$&
  ${}\simeq\ST(13)$ complex type&96\\
\midrule
$G(\mathcal I,\mathcal I,\id)$&$\id$&&28\,800\\
$G(\mathcal I,\mathcal C_2,\id)$&$\Alt(5)$&
  ${}\simeq\mathcal C_4\boxdot\mathcal I$&480\\
$G(\mathcal I,\mathcal C_2,\Theta)$&$\Alt(5)$&&480\\
$G(\mathcal I,\id,\Theta)$&$\mathcal I$&&240\\
$G(\mathcal I,\id,\rho(j))$&$\mathcal I$&
  ${}\simeq\ST(22)$ complex type&240\\
\bottomrule
\end{tabular}

\smallskip
$\delta = \frac{1}{\sqrt2}(i-j)$,\enspace $\beta\in\Aut(\mathcal O)$ is central,
\enspace $\Theta\in\Aut(\mathcal I)\setminus\Inn(\mathcal I)$
\end{table}

\begin{lemma}\label{lemma:theta}
Suppose that $H\normal\mathcal I$ and $\varphi$ is an automorphism of 
$\mathcal I/H$ of order 2 such that $G(\mathcal I,H,\varphi)$ is a reflection 
group not conjugate to $G(\mathcal I,H,\id)$. Then, up to conjugacy,
$H = \mathcal C_2$ and $\varphi = \Theta$ or $H = \id$ and 
$\varphi$ is either $\Theta$ or $\rho(j)$, where $\Theta$ is the
automorphism of $\mathcal I$ such that $\Theta(i) = -i$ and
$\Theta(\sigma) = -\tfrac12(\tau + i - \tau^{-1}k)$, where 
$\sigma = \tfrac12(\tau^{-1} + i + \tau j)$.
\end{lemma}

\begin{proof}
Suppose at first that $\varphi = \rho(\gamma)$ for some $\gamma\in\mathcal I$,
which must have order 4. If $H = \mathcal C_2$, then $\gamma^2\in H$ and 
hence $\gamma\in L_\varphi$. From Corollary~\ref{cor:id},
$G(\mathcal I,\mathcal C_2,\rho(\gamma))$ is conjugate to 
$G(\mathcal I,\mathcal C_2,\id)$, contrary to assumption. However, if $H = \id$, 
then $L_\varphi$ generates $\mathcal I$ and so $G(\mathcal I,\id,\rho(\gamma))$ is a
reflection group. There is only one class of elements of order 4 in $\mathcal I$,
therefore, by Lemma~\ref{lemma:conj}\,(ii), we may choose $\gamma = j$.

The automorphism 
$\Theta\in\Aut(\mathcal I)\simeq\Aut(\mathcal I/\mathcal C_2)\simeq\Sym(5)$
corresponds to conjugation by a transposition $t\in\Sym(5)$. If $H$ is $\id$
or $\mathcal C_2$, the image of $L_\Theta$ in the group of inner automorphisms 
$\Inn(\mathcal I)\simeq\mathcal I/\mathcal C_2\simeq\Alt(5)$ is 
$\{\, g\in \Alt(5)\mid (tg)^2 = 1\,\}$. Since $\Theta(-1) = -1$, 
$\xi\in L_\Theta$ if and only if $-\xi\in L_\Theta$. In both cases it follows 
that $|L_\Theta| = 20$. Consequently $\Theta(\xi)\xi\in\mathcal C_2$ if and 
only if $\Theta(\xi)\xi = 1$. Moreover, $L_\Theta$ generates $\mathcal I$, 
thus $G(\mathcal I,\mathcal C_2,\Theta)$ and $G(\mathcal I,\id,\Theta)$ are 
reflection groups.

Suppose that $\varphi$ is another automorphism of order 2 in
$\Aut(\mathcal I)\setminus\Inn(\mathcal I)$. Then 
$\varphi\Theta\in\Inn(\mathcal I)$ and so $\varphi\Theta = \rho(\xi)$ for 
some $\xi\in\mathcal I$. For all $\eta\in\mathcal I$ we have 
$\varphi(\eta) = \xi\Theta(\eta)\xi^{-1}$, whence 
$\eta = \varphi^2(\eta) = (\xi\Theta(\xi))\eta(\xi\Theta(\xi))^{-1}$. 
Thus $\xi\Theta(\xi)\in\mathcal C_2$ and from the previous paragraph
$\xi\Theta(\xi) = 1$. Therefore $\varphi\in L_\Theta$ and from 
Lemma~\ref{lemma:conj}\,(i), $G(\mathcal I,H,\varphi)$ is conjugate to 
$G(\mathcal I,H,\Theta)$.
\end{proof}

The results of this section are summarised in Table \ref{tbl:2}. The 
isomorphisms are described in Theorems \ref{thm:extbin} and \ref{thm:primST}.
The groups are those of Table I of \cite{cohen:1980}. Note that there is a 
misprint in that table: the order of $G(\mathcal I,\mathcal I,\id)$ should 
be 28\,800.

\section{Primitive groups with imprimitive complexification}
\label{sec:extST}

From \cite[Theorem 3.6]{cohen:1980}, if $G$ is a primitive 
quaternionic reflection group whose complexification is imprimitive, 
then the rank of $G$ is at most two.
The rank two examples can be obtained from
\begin{enumerate}[\enspace(i)]
\item a cyclic group $\mathcal C_d$.
\item binary polyhedral groups $H\normal K$, and 
\item a surjective homomorphism $\psi : \mathcal C_d\to K/H$.
\end{enumerate}
Let $\mathcal C_d\times_\psi K$ denote the pullback
\[
\begin{tikzpicture}[scale=1.25]
\node (A) at (0,1) {$\mathcal C_d\times_\psi K$};
\node (B) at (1.5,1) {$K$};
\node (C) at (0,0) {$\mathcal C_d$};
\node (D) at (1.5,0) {$K/H$};
\path[->,font=\scriptsize,>=angle 90]
(A) edge node[above]{$\pi_K$} (B)
(A) edge node[right]{$\pi_d$} (C)
(B) edge node[right]{$\mu$} (D)
(C) edge node[above]{$\psi$} (D);
\end{tikzpicture}
\]
Its elements are the pairs $(\alpha,\xi)\in\mathcal C_d\times K$ such
that $\psi(\alpha) = H\xi$.

For $\alpha\in\sS^1$, $\xi\in\sS^3$ the maps $L(\alpha) : h \mapsto \alpha h$ 
and $R(\xi) : h\mapsto h\overline{\xi}$ preserve the $\C$-structure of $\HH$.
The transformation $\theta : (\alpha,\xi)\mapsto L(\alpha)R(\xi)$ is a surjective 
homomorphism $\sS^1\times\sS^3\to\U_2(\C)$ with kernel $\{\pm(1,1)\}$. If $K$ 
is not a cyclic or binary dihedral group, then $f = |K/H|$ is 1, 2 or 3 and $f$
divides $d$. In this case we may unambiguously denote 
$\theta(\mathcal C_d\times_\psi K)\subset\U_2(\C)$ by $\mathcal C_d\circ_f K$. 
If $d$ is odd, then $\mathcal C_d\circ_f K = \mathcal C_{2d}\circ_f K$ so 
from now on we require $d$ to be even and a multiple of $f$.  Therefore 
$|\mathcal C_d\circ_f K| = \dfrac{d}{2f}|K|$.

\smallskip
In \cite[\S3]{cohen:1976} Cohen proves that the 19 primitive rank 
two Shephard and Todd groups are conjugate in $\GL_2(\C)$ to 
$\mathcal C_d\circ_f K$ where $K = \mathcal T$, $f\in\{1,3\}$ and
$d\in\{6,12\}$ or $K = \mathcal O$, $f\in\{1,2\}$ and $d\in\{4,8,12,24\}$
or $K = \mathcal I$, $f = 1$ and $d\in\{4,6,10,12,20,30,60\}$. 

\begin{lemma}\label{lemma:s}
If $G\subset\U_2(\C)$, the matrix 
$s = \begin{pmatrix}\phantom{-}0&j\\-j&0\end{pmatrix}$ normalises the 
quaternionic group $G^\sharp$ of complex type if and only if 
$\det(g)^{-1}g\in G$ for all $g\in G$. In particular, if $K$ is 
$\mathcal T$, $\mathcal O$ or $\mathcal I$ and $d$ is even, all 
groups $(\mathcal C_d\circ_f K)^\sharp$ except 
$(\mathcal C_d\circ_3\mathcal T)^\sharp$ are normalised by $s$.
\end{lemma}

\begin{proof}
In general, if $g\in \U_2(\C)$, then 
$sg^\sharp s^{-1} = \overline{\det(g)}g^\sharp = \det(g)^{-1}g^\sharp$.
For the groups $G = (\mathcal C_{2d}\circ_f K)^\sharp$ with $f\in \{1,2\}$
the determinant of every element is a power of $\zeta_{2d}^2$ and 
$\zeta_{2d}^2\id\in G$. However, if $K = \mathcal T$ and $f = 3$ there 
exists $\varpi\in K$ of order 3 such that 
$g = \theta(\zeta_d,\varpi)\in\mathcal C_d\circ_3\mathcal T$. Its determinant
is $\zeta_d^2$ but $\zeta_d^{-2}g \notin\mathcal C_d\circ_3\mathcal T$
since $\zeta_d^2\id\notin\mathcal C_d\circ_3\mathcal T$.
\end{proof}

\begin{defn}
Suppose that $K$ is $\mathcal T$, $\mathcal O$ or $\mathcal I$, $d$ is even
and $f$ is 1 or 2.  The \emph{extended binary polyhedral group} 
$\mathcal C_d\boxdot_f K$ is the semidirect product 
$(\mathcal C_d\circ_f K)^\sharp\(s\)$ where $s$ is defined in Lemma 
\ref{lemma:s}. If $f = 1$, 
we omit the subscript $f$. From its construction, the complexification 
$(\mathcal C_d\boxdot_f K)^\circ$ is imprimitive.
\end{defn}
 
\begin{thm}\label{thm:extST}
The extended binary polyhedral group $\mathcal C_d\boxdot_f K$ is 
generated by its quaternionic reflections if and only if it is one of 
the following.
\begin{enumerate}[\enspace(i)]
\item $\mathcal C_d\boxdot\mathcal T$, order $24d$ with $d$ a multiple of 6.
\label{i:a}
\item $\mathcal C_d\boxdot\mathcal O$, order $48d$ with $d$ a multiple of 4.
\label{i:b}
\item $\mathcal C_d\boxdot_2\mathcal O$, order $24d$ with $d$ a multiple 
of 4 but not divisible by 16.\label{i:c}
\item $\mathcal C_d\boxdot\mathcal I$, order $120d$ with $d$ a multiple of 4, 
6 or 10.\label{i:d}
\end{enumerate}
The groups are primitive except for $\mathcal C_4\boxdot\mathcal O$,
$\mathcal C_4\boxdot_2\mathcal O$ and $\mathcal C_4\boxdot\mathcal I$
(see Theorem \ref{thm:extbin}).
\end{thm}

\begin{proof}
The matrices $s = \begin{pmatrix}0&j\\-j&0\end{pmatrix}$ and 
$z = \theta(\zeta_d,1) = \begin{pmatrix}\zeta_d&0\\0&\zeta_d\end{pmatrix}$ 
generate a dihedral group $D$ of order $2d$ and they commute with 
$\theta(1,\xi)$ for all $\xi\in K$. If $f = 1$, then for each divisor 
$e$ of $d$, we have $\mathcal C_d\boxdot K = D\circ(\mathcal C_e\circ K)$.  
Therefore $\mathcal C_d\boxdot K$ is generated by quaternionic reflections 
if and only if $\mathcal C_e\circ K$ is a complex reflection group for some 
divisor $e$ of~$d$. Thus cases \eqref{i:a}, \eqref{i:b} and \eqref{i:d} 
follow from \cite[\S3]{cohen:1976}.

From \cite[Theorem 5.14]{lehrer-taylor:2009} we may suppose that the image
of $\mathcal O$ in $\SU_2(\C)$ is generated by
\[
  g_1 = \tfrac12\begin{pmatrix}-1+i&\phantom{-}1+i\\-1+i&-1-i\end{pmatrix},
  \enspace\text{and}\enspace
  g_2 = \tfrac{1}{\sqrt2}\begin{pmatrix}1+i&0\\0&1-i\end{pmatrix}.
\]
and $\mathcal T$ is the subgroup $\(g_1,g_2^2\)$. Then 
$\mathcal C_d\circ\mathcal O = \(z,g_1,g_2\)$,
$\mathcal C_d\boxdot\mathcal O = \(z,g_1,g_2,s\)$ and
$\mathcal C_d\boxdot_2\mathcal O = \(z^2, g_1, g_2^2, zg_2, s\)$. The dihedral 
group $E = \(z^2,s\)$ is a subgroup of $\mathcal C_d\boxdot_2\mathcal O$.

If $d = 4e$ and $e$ is odd, then $\mathcal C_4\circ_2\mathcal O =
\(z^{2e}, g_1, g_2^2, z^eg_2\)\simeq\ST(12)$ and $\mathcal C_d\boxdot_2\mathcal 
O = E\circ (\mathcal C_4\circ_2\mathcal O)$. Similarly, if $d = 8e$ and $e$ is odd, 
then $\mathcal C_8\circ_2\mathcal O = \(z^{2e}, g_1, g_2^2, z^eg_2\)\simeq\ST(8)$
and $\mathcal C_d\boxdot_2\mathcal O = E\circ (\mathcal C_8\circ_2\mathcal O)$. 
On the other hand, if $d = 16e$, it follows from the description in
\cite[Ch.~5, \S3]{lehrer-taylor:2009} that there are only 6 reflections in 
$\mathcal C_d\circ_2\mathcal O$ and they generate $\ST(4,2,2)$. The other 
reflections in $\mathcal C_d\boxdot_2\mathcal O$ are the $d/2$ reflections in 
$E$; therefore $\mathcal C_{16e}\boxdot_2\mathcal O$ is not generated by
reflections. This completes the proof of~\eqref{i:c}.
\end{proof}

\begin{rem}
Lemma (3.3) of \cite{cohen:1980} states that if $G\subset\U_2(\HH)$ is a 
rank two primitive quaternionic reflection group such that its 
complexification is imprimitive and not monomial, then $G$ is conjugate
to one of the groups listed in Theorem \ref{thm:extST}.

However, in cases \eqref{i:b}, \eqref{i:c} and \eqref{i:d} with $d = 4$,
the group is imprimitive and conjugate to a group in Table \ref{tbl:2}. 
(The assumption in \cite[Lemma (3.2)]{cohen:1980} that a primitive complex
reflection group $H$ remains primitive as a quaternionic group $H^\sharp$
of complex type does not hold. See Theorem \ref{thm:boxdot} below for the 
details.)
\end{rem}

\section{Systems of imprimitivity}\label{sec:imprim}
\begin{thm}
If standard copies $G_1 = G(K_1,H_1,\psi_{r_1})$ and 
$G_2 = G(K_2,H_2,\psi_{r_2})$ are conjugate in $\U_2(\HH)$ and they each 
have a unique system of imprimitivity, then $K_1 = K_2$, $H_1 = H_2$ and
$r_1 = r_2$. 
\end{thm}

\begin{proof}
Suppose that $g^{-1}G_1g = G_2$. By uniqueness, $g$ fixes the standard
system of imprimitivity. Thus $g$ is monomial.
As in the proof of \cite[Lemma 2.4]{cohen:1980}, it follows that
$g^{-1}K_1g = K_2$, $g^{-1}H_1g = H_2$ and $r_1 = r_2$. From the
assumption that $G_1$ and $G_2$ are standard copies (see Definitions 
\ref{defn:stdDC}, \ref{defn:stdDD}) it follows that $K_1 = K_2$ and $H_1 = H_2$.
\end{proof}

\begin{lemma}\label{lemma:systems}
Let $G = G(K,H,\varphi)$ and suppose that $\soi{u,v}\ne\soi{e_1,e_2}$ is a 
system of imprimitivity for $G$. Then $|H|\le 2$ and we may suppose that 
either $u = (1,1)$, $v = (1,-1)$ or $u = (1,\theta)$, $v = (\theta,1)$,
$\theta\ne 0$ and $\theta+\overline\theta = 0$. If $|H| = 2$, then 
$\theta^2 = -1$. Furthermore, for all $\eta\in K$ either 
$\theta^{-1}\eta\theta\in\varphi(H\eta)$ or
$-\theta^{-1}\eta\theta\in\varphi(H\eta)$.
\end{lemma}

\begin{proof}
The matrices $s = \begin{pmatrix}0&1\\1&0\end{pmatrix}$ and
$f = \begin{pmatrix}\xi&0\\0&1\end{pmatrix}$ are elements of $G$ (when $\xi\in H$).

By replacing $u$ with a scalar multiple we may suppose that $u = (1,\theta)$ 
for some $\theta\ne 0$.  If $us = (\theta,1)$ is a multiple of $u$, then 
$\theta = \pm1$ and we may take $u = (1,1)$ and $v = (1,-1)$. On the other 
hand, if $\theta\ne \pm$1, then $v = us = (\theta,1)$ is orthogonal to $u$ and 
hence $\theta + \overline\theta = 0$.

If $u = (1,1)$ and $uf$ is a multiple of $u$, then $\xi = 1$; if $uf$ is a
multiple of $(1,-1)$, then $\xi = -1$.  If $u = (1,\theta)$ for some 
$\theta\ne\pm 1$ and $uf$ is a multiple of $u$, then $\xi = 1$. However, 
if $uf$ is a multiple of $(\theta,1)$, then $\xi = \theta^2$ is uniquely 
determined by $\theta$ and therefore $\xi = -1$ since $\theta\ne\pm 1$. 
Thus $|H|\le 2$ and $|H| = 2$ implies $\theta^2 = -1$.

Suppose that $u = (1,\theta)$, $v = (\theta,1)$, $\theta\ne 0$ and 
$\theta+\overline\theta = 0$. If 
$\eta\in K$ and $\lambda\in\varphi(H\eta)$, then
$g = \begin{pmatrix}\eta&0\\0&\lambda\end{pmatrix}\in G$.
If $ug$ is a multiple of $u$, then $\theta^{-1}\eta\theta = \lambda$. 
However, if $ug$ is a multiple of $v$, then $\theta\lambda\theta = \eta$
and on taking norms we find that $\Norm(\theta)^2 = 1$; therefore 
$\theta^2 = -1$ and so $\theta^{-1}\eta\theta = -\lambda$. This completes
the proof.
\end{proof}

\begin{cor}
If $G(K,H,\varphi)$ has more then one system of imprimitivity, then every
reflection has order 2.
\end{cor}

\subsection{The binary dihedral groups.}
\begin{thm}\label{thm:bin_dihedral}
Suppose that $\soi{u,v}\ne\soi{e_1,e_2}$ is a system of imprimitivity for
the standard copy 
$G = G(\mathcal D_m,\mathcal C_\ell,\psi_r)$. Then $v$ is orthogonal to $u$, 
$\ell$ is 1 or 2 and the possible representatives for $u$ are $(1,1)$, 
$(1,ai)$, $(1,j)$ and $(1,ck)$ for some non-zero $a, c\in\R$.
\begin{enumerate}[\enspace(i)]
\item $\soi{(1,1),(1,-1)}$ is a system of imprimitivity for $G$
if and only if $r = 1$.
\item If $\soi{(1,ai),(ai,1)}$ is a system of imprimitivity for $G$, then
$r = 1$. If $\ell = 1$, it is a system of imprimitivity for all $m$; if
$\ell = 2$, it is a system of imprimitivity for $G$ if and only if 
$a \in\{-1,0,1\}$.
\item
If $\ell = 1$, then $\soi{(1,j),(j,1)}$ is a system of 
imprimitivity for $G$ if and only if $m\equiv 2\pmod 4$ and $r = m-1$ or 
$m = r = 1$. If $\ell = 2$, it is a system of imprimitivity for $G$ if and 
only if $m \in \{1,2\}$ and $r = 1$.
\item 
$\soi{(1,ck),(ck,1)}$ is a system of imprimitivity for $G$ for all $c\in\R$
if and only if $m = \ell = 1$. Otherwise, if $\ell = 1$ and $m \ge 2$, 
$\soi{(1,k),(k,1)}$ is a system of imprimitivity for $G$ if and only if 
$m\equiv 2\pmod 4$ and  $r = m-1$. If $\ell = 2$, $\soi{(1,k),(k,1)}$ is 
a system of imprimitivity for $G$ if and only if $m \in \{1,2\}$ and $r = 1$.
\end{enumerate}
\end{thm}

\begin{proof}
From the previous lemma $\ell\le 2$. It follows from \eqref{eq:std-gens} with 
$x = \eta = \zeta_{2m}$ and $y = j$ that $G(\mathcal D_m,\mathcal C_2,\psi_r)$
is generated by
\begin{equation}\label{eq:dc-gens}
  s = \begin{pmatrix}0&1\\1&0\end{pmatrix},\quad
  f = \begin{pmatrix}-1&0\\0&1\end{pmatrix},\quad
  g = \begin{pmatrix}\eta&0\\0&\eta^r\end{pmatrix},\quad
  h = \begin{pmatrix}j&0\\0&-j\end{pmatrix}
\end{equation}
and $G(\mathcal D_m,\id,\psi_r)$ is the subgroup generated by $s$, 
$g$ and $h$.

\medskip
\noindent (i)\enspace
If $u = (1,1)$, then $ug$ is a multiple of $u$ or $v = (1,-1)$ if and
only if $\eta^{r-1} = \pm1$. Therefore $r$ is 1 or $m+1$. But $r = m+1$ 
is excluded by Theorem \ref{thm:A} (we assume that $r \le m/2$).
If $\ell = 1$, the group is generated by $s$, $g$ and $h$.  Thus the 
same argument implies $r = 1$. This completes the proof of (i).

\smallskip
Suppose that $u = (1,\theta)$, $v = (\theta,1)$, $\theta\ne 0$ and 
$\theta+\overline\theta = 0$. Writing $\theta = ai+bj+ck$ we have 
\begin{align}\label{eq:i}
  \theta i\theta &= (-a^2+b^2+c^2)i - 2abj - 2ack\enspace\text{and}\\
  \label{eq:j}
  \theta j\theta &= -2abi + (a^2-b^2+c^2)j - 2bck.
\end{align}
We also have $-\theta^2 = \Norm(\theta) = a^2 + b^2 + c^2$. It follows 
from Lemma \ref{lemma:systems} that $\theta^{-1}\eta\theta = \pm \eta^r$
and so $\theta^{-1}i\theta$ is a real linear combination of 1 and $i$.
Therefore, from \eqref{eq:i} we have $ab = ac = 0$.  Furthermore, 
$\theta^{-1}j\theta = \pm j$ and from \eqref{eq:j} we have $ab = bc = 0$.
Thus $\theta$ is $ai$, $bj$ or $ck$. If $uf$ is a multiple of $u$, then
$\theta = 0$. But then $u$ and $v$ are the standard basis vectors.  If $uf$
is a multiple of $v$, then $\theta^2 = -1$. Consequently, if $\ell = 2$,
then $\theta\in\{\pm i,\pm j,\pm k\}$ and interchanging $u$ and $v$, if
necessary, we may suppose that $\theta\in\{i,j,k\}$.

\medskip
\noindent (ii)\enspace
If $u = (1,\theta)$ and $\theta = ai$, it follows from Lemma \ref{lemma:systems}
that $\pm\eta^r = \theta^{-1}\eta\theta = \eta$, hence $r = 1$. The matrices $s$,
$g$ and $h$ of \eqref{eq:dc-gens} generate $G(\mathcal D_m,\id,\psi_r)$
and a straightforward calculation shows that they preserve $\soi{u,v}$.

We have shown above that $\theta = \pm i$ when $\ell = 2$, and we may choose 
$\theta = i$. In this case $uf$ is a multiple of $u$.  This completes the
proof of (ii).

\medskip
\noindent (iii)\enspace
Suppose that $u = (1,\theta)$ and $\theta = bj$. If $uh$ is a multiple of $u$, 
then $b = 0$ and $u, v$ is the standard basis. If $uh$ is a multiple of $v$, 
then $b = \pm 1$. Thus we may suppose that $u = (1,j)$.

From Lemma \ref{lemma:systems} we have $\pm\eta^r = \theta^{-1}\eta\theta$
and in this case $\theta^{-1}\eta\theta = \eta^{-1}$. Therefore either 
$\eta^r = \eta^{-1} = \eta^{2m-1}$ or $\eta^r = \eta^{m-1}$, whence $r = 2m-1$
or $r = m-1$. However, it follows from Theorem \ref{thm:A} that either $r = 1$ 
or $r\le n/2$, where $n = 2m/\ell$. Therefore, if $\ell = 2$, the groups 
$G(\mathcal D_1,\mathcal C_2,\psi_1)$ and $G(\mathcal D_2,\mathcal C_2,\psi_1)$ 
are the only possibilities. It can be checked directly that $\soi{(1,j),(j,1)}$ 
is a system of imprimitivity for both groups.

If $\ell = 1$ and $r = 1$, then $m\in\{1,2\}$ and 
$G(\mathcal D_1,\id,\psi_1)$ and $G(\mathcal D_2,\id,\psi_1)$ 
are the only possibilities. This leaves the case where $\ell = 1$ and $m > 2$.
Then $r = m-1$ and from Theorem \ref{thm:A} we require $\gcd(\kappa,\nu) = 1$, 
where $\kappa = 2m/\gcd(2m,r+1)$ and $\nu = 2m/\gcd(2m,r-1)$.  This holds if
and only if $m \equiv 2\pmod 4$.

\medskip
\noindent (iv)\enspace
Suppose that $u = (1,\theta)$ and $\theta = ck$. We have 
$\pm\eta^r = \theta^{-1}\eta\theta = \eta^{-1}$. Thus $r = 1$ implies
$\eta^2 = \pm 1$ and hence $m\le 2$. If $\ell = 2$, the argument of 
part~(iii) shows that $G(\mathcal D_1,\mathcal C_2,\psi_1)$ and 
$G(\mathcal D_2,\mathcal C_2,\psi_1)$ are the only possibilities.  
However $\soi{(1,ck),(ck,1)}$ is a system of imprimitivity for 
$G(\mathcal D_1,\id,\psi_1)$ without restriction on~$c$ whereas 
$G(\mathcal D_2,\id,\psi_1)$ has a system of 
imprimitivity of this form if and only if $c\in\{-1, 0, 1\}$.

Suppose that $m > 2$. As in the proof of part (iii) we deduce from 
Lemma \ref{lemma:systems} that $r = 2m-1$ or $r = m-1$. Thus there
are no further possibilities when $\ell = 2$.

Finally, if $\ell = 1$ and $m > 2$, then $r = m-1$ and $\soi{(1,k),(k,1)}$
is a system of imprimitivity for $G(\mathcal D_m,\id,\psi_{m-1})$ 
if and only if $m \equiv 2 \pmod 4$.
\end{proof}	

Table \ref{tbl:3} summarises these results.
See Theorems \ref{thm:ST} and \ref{thm:DmC2} for the isomorphisms 
$G(\mathcal D_m,\id,\psi_1)\simeq\ST(2m,m,2)$ and 
$G(\mathcal D_{2m},\id,\psi_{2m-1})\simeq G(\mathcal D_m,\mathcal C_2,\psi_1)$, 
$m$ odd.

\begin{table}[ht]
\caption{Systems of imprimitivity $\soi{u,v}$ for rank two quaternionic 
reflection groups}
\label{tbl:3}
\smallskip
\centering
\begin{tabular}{l|l|l}
&\small Representatives for $u\ne (1,0)$\\[2pt]
\toprule
$G(\mathcal D_1,\id,\psi_1)$& $(1,1)$, $(1,ai)$, $(1,j)$, $(1,ck)$&
  $\simeq\ST(2,1,2)$\\
$G(\mathcal D_2,\id,\psi_1)$& $(1,1)$, $(1,ai)$, $(1,j)$, $(1,k)$&
  $\simeq\ST(4,2,2)$\\
$G(\mathcal D_1,\mathcal C_2,\psi_1)$& $(1,1)$, $(1,i)$, $(1,j)$, $(1,k)$&
  $\simeq\ST(4,2,2)$\\
$G(\mathcal D_2,\mathcal C_2,\psi_1)$& $(1,1)$, $(1,i)$, $(1,j)$, $(1,k)$\\
\midrule
$G(\mathcal D_m,\id,\psi_1)$&$(1,1)$, $(1,ai)$&$\simeq\ST(2m,m,2)$\\
$G(\mathcal D_m,\mathcal C_2,\psi_1)$&$(1,1)$, $(1,i)$&$m > 2$\\
$G(\mathcal D_{2m},\id,\psi_{2m-1})$&$(1,j)$, $(1,k)$&
  $\simeq G(\mathcal D_m,\mathcal C_2,\psi_1)$, $m$ odd\\
\bottomrule
\end{tabular}

\smallskip
$a, c\in \R$
\end{table}

\begin{rem}
The matrix $T$ defined in \eqref{eq:RST} is a reflection. It normalises
$G(\mathcal D_m,\mathcal C_2,\psi_1)$ and preserves $\soi{(1,i),(i,1)}$.
Thus $G(\mathcal D_m,\mathcal C_2,\psi_1)\(T\)$ is imprimitive and
isomorphic to
$G(\mathcal D_m,\mathcal C_4,\psi_1)$,
$G(\mathcal D_{4m},\id,\psi_{2m-1})$,
$G(\mathcal D_{2m},\mathcal C_2,\psi_{m-1})$ or
$G(\mathcal D_{4m},\id,\psi_{2m+1})$ 
for $m \equiv 0,1,2,3\pmod 4$. 

\end{rem}

\subsection{The binary tetra-, octa- and icosahedral groups.}
Throughout this section $\delta = \frac{1}{\sqrt2}(i-j)$, $\beta$
is the central automorphism of $\mathcal O$ and $\Theta$ is
an outer automorphism of $\mathcal I$ defined in Lemma \ref{lemma:theta}.

If $G(K,H,\varphi)$ is a quaternionic reflection group with a system 
of imprimitivity $\Sigma$ in addition to $\soi{e_1,e_2}$, it follows
from Lemma \ref{lemma:systems} $H$ is $\id$ or $\mathcal C_2$ and
$\Sigma = \soi{(1,1),(1,-1)}$ or $\soi{(1,\theta),(\theta,1)}$, where
$\theta\ne 0$ and $\theta+\overline\theta = 0$.

We consider the groups from Table \ref{tbl:2} with $|H|\le 2$ and use 
the matrices
\begin{equation}\label{eq:mats}
  s = \begin{pmatrix}0&1\\1&0\end{pmatrix},\quad
  f = \begin{pmatrix}-1&0\\0&1\end{pmatrix},\quad
  g_\xi = \begin{pmatrix}\xi&0\\0&\varphi(\xi)\end{pmatrix}\enspace (\xi\in K).
\end{equation}
If $H = \mathcal C_2$, then $G$ is generated by $s$, $f$ and $g_\xi$ for 
$\xi$ in a set of generators of $K$. (If $H = \id$, omit $f$.)

\begin{thm}\label{thm:binpol_11}
Suppose that $K$ is $\mathcal T$, $\mathcal O$ or $\mathcal I$.
Then $\Sigma = \soi{(1,1),(1,-1)}$ is a system of imprimitivity for a
quaternionic reflection group $G = G(K,H,\varphi)$ if and only if $G$ is 
\[
  G(\mathcal O,\mathcal C_2,\id),\enspace G(\mathcal O,\id,\beta)\enspace
  \text{or}\enspace G(\mathcal I,\mathcal C_2,\id).
\]
\end{thm}

\begin{proof}
We have $\rho(\delta)(i) = -j$ and 
$\Theta(j) = -k$. Therefore $\begin{pmatrix}i&0\\0&-j\end{pmatrix}$ is
contained in $G(\mathcal T,\mathcal C_2,\rho(\delta))$, 
$G(\mathcal T,\id,\rho(\delta))$ and $G(\mathcal O,\id,\rho(\delta))$
whereas $\begin{pmatrix}j&0\\0&-k\end{pmatrix}$ belongs to
$G(\mathcal I,\mathcal C_2,\Theta)$ and $G(\mathcal I,\id,\Theta)$.
Furthermore $\varpi = \tfrac12(-1 + i + j + k)\in\mathcal I$,  
$\rho(j)(\varpi) = \varpi i$ and so
$\begin{pmatrix}\varpi&0\\0&\varpi i\end{pmatrix}$ is in
$G(\mathcal I,\id,\rho(j))$.
Consequently none of these groups preserve $\Sigma$.

If $\varphi = \id$ or $\varphi = \beta$ it is clear that the matrices of \eqref{eq:mats} preserve $\Sigma$. Therefore $\Sigma$ is a system of
imprimitivity for $G(\mathcal O,\mathcal C_2,\id)$, 
$G(\mathcal O,\id,\beta)$ and $G(\mathcal I,\mathcal C_2,\id)$.
\end{proof}

\begin{thm}\label{thm:binpol_theta}
Suppose that $K$ is $\mathcal T$, $\mathcal O$ or $\mathcal I$. If 
$\theta\ne 0$ and $\theta+\overline\theta=0$,
then $\soi{(1,\theta),(\theta,1)}$ is a system of imprimitivity 
for a quaternionic reflection group $G = G(K,H,\varphi)$ from 
Table \ref{tbl:2} if and only if one of the following holds.
\begin{enumerate}[\enspace(i)]
\item $\theta = \delta$ and $G$ is $G(\mathcal T,\mathcal C_2,\rho(\delta))$.
\item $\theta = r\delta$ for some $r\in\R$ and $G$ is 
$G(\mathcal T,\id,\rho(\delta))$ or $G(\mathcal O,\id,\rho(\delta))$.
\item $\theta = rj$ for some $r\in\R$ and $G$ is $G(\mathcal I,\id,\rho(j))$.
\end{enumerate}
\end{thm} 

\begin{proof}
Let $u = (1,\theta)$ and $v = (\theta,1)$ and use the fact that 
$i$, $j$ and $\varpi$ belong to $\mathcal T$, $\mathcal O$ and~$\mathcal I$.
If $ug_\xi$ is a multiple of $u$, then $\theta^{-1}\xi\theta = \varphi(\xi)$ 
and hence $\theta\xi\theta = -\Norm(\theta)\varphi(\xi)$. If $ug_\xi$ is a 
multiple of $v$, then $\theta\varphi(\xi)\theta = \xi$, hence 
$\Norm(\theta)^2 = 1$ and so $\theta^2 = -1$. Thus
\begin{equation}\label{eq:theta}
  \theta\xi\theta = -\Norm(\theta)\varphi(\xi)\quad\text{or}\quad
  \theta\xi\theta = \varphi(\xi).
\end{equation}

\noindent
\textbf{Suppose that} $\varphi = \rho(\delta)$. 
Since $\delta i\delta^{-1} = -j$, \eqref{eq:theta} shows that
$\theta i\theta$ is a multiple of $j$ and $\theta j\theta$ is a multiple 
of $i$. On writing $\theta = ai+bj+ck$ for some $a,b,c\in\R$ it follows
from \eqref{eq:i} and \eqref{eq:j} that $c = 0$ and $a^2 = b^2$, 
whence $\theta = ai + bj$ and $\Norm(\theta) = 2a^2$.

A simple calculation shows that $\delta\varpi\delta^{-1} = \varpi^{-1}$ 
and $\theta\varpi\theta = a^2(1 - (b/a)i - (b/a)j + k)$. If 
$\theta\varpi\theta = \delta\varpi\delta^{-1}$ we obtain a contradiction. 
Therefore $\theta\varpi\theta = -\Norm(\theta)\delta\varpi\delta^{-1}$, 
which implies $b = -a$.  This proves that $\theta = r\delta$ for
some $r\in\R$. If $H = \id$, then direct calculation with the
matrices $s$ and $g_\xi$ of \eqref{eq:mats} shows that 
$\soi{(1,r\delta),(r\delta,1)}$ is a system of imprimitivity of
$G(\mathcal T,\id,\rho(\delta))$ and $G(\mathcal O,\id,\rho(\delta))$
for all $r$. However, if $H = \mathcal C_2$, the action of the
generator $f$ of $H$ on $u$ shows that $\theta^2 = -1$ and in this 
case the only possibility is $\soi{(1,\delta),(\delta,1)}$.

\smallskip
\noindent
\textbf{Suppose that} $\varphi$ is $\id$ or $\beta$. Then $K$ is
$\mathcal O$ or $\mathcal I$ and $\mathcal T\subset K$. 
In this case, for $\xi\in\mathcal T$, \eqref{eq:theta} implies 
$\theta\xi\theta$ is $-\Norm(\theta)\xi$ or $\xi$. Writing 
$\theta = ai+bj+ck$ as above, then applying \eqref{eq:i} and 
\eqref{eq:j} to $i$ and $j$ we find that $\theta$ is $ai$, $bj$ or 
$ck$. But then $\theta\varpi\theta$ is not a multiple of $\varpi$,
which is a contradiction.

\smallskip
\noindent
\textbf{Suppose that} $\varphi = \rho(j)$. Again using \eqref{eq:theta}
and applying \eqref{eq:i} and \eqref{eq:j} to $i$ and $j$ we find that 
$\theta$ is $ai$, $bj$ or $ck$.  An easy calculation shows that when 
$\theta$ is $ai$ or $ck$, $\theta\varpi\theta$ is not a multiple of 
$\rho(j)(\varpi)$.  This leaves $bj$ as the only possibility. Conversely,
$\soi{(1,bj),(bj,1)}$ is a system imprimitivity of 
$G(\mathcal I,\id,\rho(j))$ for all $b\in\R$.

\smallskip
\noindent
\textbf{Suppose that} $\varphi = \Theta$. In this case,
$\Theta(i) = -i$ and $\Theta(j) = -k$.
From \eqref{eq:i} and \eqref{eq:j} we deduce that $\theta = bj+ck$
and $b^2 = c^2$. We have $\theta\varpi\theta = b^2(1+i-(b/c)j-(b/c)k)$
and $\Theta(\varpi) = \varpi^{-1}$. From this we deduce that $b = -c$.
But then, for the generator $\sigma$ of $\mathcal I$, we obtain the
contradiction that $(j-k)\sigma(j-k) = -\tau^{-1} + i +\tau k$ is not a
multiple of~$\Theta(\sigma)$.
\end{proof}

\begin{table}[ht]
\caption{Systems of imprimitivity $\soi{u,v}$ for $G(K,H,\varphi)$ 
where $K$ is $\mathcal T$, $\mathcal O$ or $\mathcal I$}
\label{tbl:4}
\smallskip
\centering
\begin{tabular}{l|c|l}
&\small Representatives for $u\ne (1,0)$\\[2pt]
\toprule
$G(\mathcal T,\mathcal C_2,\rho(\delta))$&$(1,\delta)$&
  ${}\simeq G(\mathcal O,\id,\beta)\simeq\mathcal C_4\boxdot_2\mathcal O$\\
$G(\mathcal T,\id,\rho(\delta))$&$(1,r\delta)$&
  ${}\simeq\ST(12)\simeq\GL(2,3)$\\
$G(\mathcal O,\mathcal C_2,\id)$&$(1,1)$&
  ${}\simeq\mathcal C_4\boxdot\mathcal O$\\
$G(\mathcal O,\id,\beta)$&$(1,1)$&
  ${}\simeq G(\mathcal T,\mathcal C_2,\rho(\delta))\simeq
  \mathcal C_4\boxdot_2\mathcal O$\\
$G(\mathcal O,\id,\rho(\delta))$&$(1,r\delta)$&${}\simeq\ST(13)$\\
$G(\mathcal I,\mathcal C_2,\id)$&$(1,1)$&
  ${}\simeq\mathcal C_4\boxdot\mathcal I$\\
$G(\mathcal I,\id,\rho(j))$&$(1,rj)$&${}\simeq\ST(22)$\\
\bottomrule
\end{tabular}

\smallskip
\smallskip
$\delta = \frac{1}{\sqrt2}(i-j)$,\enspace $\beta\in\Aut(\mathcal O)$ is central,
\enspace $r\in \R$
\end{table}

\subsection{The extended binary polyhedral groups.}
\begin{thm}\label{thm:extbin}
Suppose that $G = \mathcal C_d\boxdot_f K$ is a quaternionic 
reflection group where $K$ is $\mathcal T$, $\mathcal O$ or $\mathcal I$ 
and $H = \mathcal C_d\circ_f K$ is primitive. Then $\Sigma$ is a 
system of imprimitivity for $G$ if and only if $\Sigma$ is 
$\soi{(1,j),(j,1)}$ or $\soi{(1,k),(k,1)}$ and $G$ is
$\mathcal C_4\boxdot\mathcal O$, $\mathcal C_4\boxdot_2\mathcal O$
or $\mathcal C_4\boxdot\mathcal I$.
\end{thm}

\begin{proof}
In all cases $G$ contains the matrices
\[
  g = \begin{pmatrix}\zeta_d&0\\0&\zeta_d\end{pmatrix},\enspace
  h = \begin{pmatrix}i&0\\0&-i\end{pmatrix}\in K^\sharp\enspace
  \text{and}\enspace
  s = \begin{pmatrix}0&j\\-j&0\end{pmatrix}.
\]

We may suppose that $\Sigma = \soi{u,v}$ where $u = (1,\theta)$
and where $\theta\notin\C$ because we assume that $H$ is primitive.
We may scale $v$ so that $v = (-\overline\theta,1)$.

Write $\theta = a+bi+cj+dk$. If $uh$ is a multiple of $v$, then
$\theta i\overline\theta = i$ and therefore $n\theta i = i\theta$, 
where $n = \Norm(\theta) > 0$. But then $a(n-1) = b(n-1) = c(n+1) = d(n+1) = 0$,
and so $c = d = 0$. Therefore $\theta\in\C$, contrary to assumption.
Consequently $uh$ is a multiple of $u$, whence $i\theta = -\theta i$
and $a = b = 0$.

If $us$ is a multiple of $u$, then $\theta j\theta = -j$, whereas,
if it is a multiple of $v$, then $j\overline\theta = \theta j$ and
hence $\theta j\theta = j\Norm(\theta)$. Since
$\theta j\theta = (d^2 - c^2)j -2cd k$ we find that $cd = 0$. That is,
$\theta$ is $cj$ or $dk$.

If $ug$ where a multiple of $u$, then $\zeta_d\theta = \theta\zeta_d$
and we would have $d = 2$, which contradicts Theorem \ref{thm:extST}.  
Thus $ug$ is a multiple of $v$ and so
$-\theta\zeta_d\overline\theta = \zeta_d$. Therefore $\Norm(\theta)^2 = 1$, 
$\zeta_d = \pm i$ and $d = 4$, whence $K\ne\mathcal T$. Consequently 
$\Sigma$ is $\soi{(1,j),(j,1)}$ or $\soi{(1,k),(k,1)}$. Conversely, it
can be checked directly that these values of $\Sigma$ are systems of
imprimitivity for $\mathcal C_4\boxdot\mathcal O$, 
$\mathcal C_4\boxdot_2\mathcal O$ and $\mathcal C_4\boxdot\mathcal I$.
\end{proof}

\section{Conjugacy}\label{sec:conj}
For $r\in\R$ and $\theta\in\HH$ such that $\theta^2 = -1$, define
\begin{equation}\label{eq:RST}
  R_{r,\theta} = \frac{1}{\sqrt{1+r^2}}\begin{pmatrix}
    1&r\theta\\-r\theta&-1\end{pmatrix}\quad
  \text{and}\quad
  T = \tfrac{1}{\sqrt2}\begin{pmatrix}1&\phantom{-}1\\1&-1\end{pmatrix}.
\end{equation}

These are reflections of order 2 whose rows are systems of imprimitivity 
for the quaternion reflection groups of Theorems \ref{thm:bin_dihedral},
\ref{thm:binpol_11}, \ref{thm:binpol_theta} and \ref{thm:extbin}. In some 
cases, as illustrated by the following theorem, a matrix of \eqref{eq:RST}
corresponding to
a system of imprimitivity of $G = G(K,H,\varphi)$ conjugates $G$ to
$G(K',H',\varphi')$ with $K' \ne K$. In other cases the matrices 
corresponding to the systems of imprimitivity belong to the normaliser
of $G$.

\subsection{Conjugacy in the binary tetra-, octa- and icosahedral families.}
\begin{thm}\label{thm:TO}
The groups $G(\mathcal T,\mathcal C_2,\rho(\delta))$ and
$G(\mathcal O,\id,\beta)$ are conjugate in $\U_2(\HH)$.
\end{thm}

\begin{proof}
The group $G(\mathcal T,\mathcal C_2,\rho(\delta))$ is generated by
\[
  \begin{pmatrix}0&1\\1&0\end{pmatrix},\quad
  \begin{pmatrix}-1&0\\0&1\end{pmatrix},\quad
  \begin{pmatrix}i&0\\0&-j\end{pmatrix}\enspace\text{and}\enspace
  \begin{pmatrix}\varpi&0\\0&\varpi^{-1}\end{pmatrix}.
\]
The reflection $R_{1,\delta}$ conjugates these matrices to generators of
$G(\mathcal O,\id,\rho(\delta)^{-1}\beta)$.
Then $\delta\in L_\varphi$ and so $G(\mathcal O,\id,\beta)$ is the result of 
a further conjugation by $\begin{pmatrix}1&0\\0&\delta\end{pmatrix}$.
\end{proof}

\begin{rem}
This is a counterexample to Lemma (2.3) of \cite{cohen:1980}.
\end{rem}

\begin{thm}\label{thm:primST}\relax\leavevmode
\begin{enumerate}[\enspace(i)]
\item 
The groups $G(\mathcal T,\id,\rho(\delta))$, $G(\mathcal O,\id,\rho(\delta))$
and $G(\mathcal I,\id,\rho(j))$ are conjugate in $\U_2(\HH)$ to the
Shephard and Todd groups $\ST(12)$, $\ST(13)$ and $\ST(22)$ of complex type.
\item
For all $r\in\R$ the reflections $R_{r,\delta}$ normalise 
$G(\mathcal T,\id,\rho(\delta))$ and $G(\mathcal O,\id,\rho(\delta))$ and
the reflections $R_{r,j}$ normalise $G(\mathcal I,\id,\rho(j))$.
\end{enumerate}
\end{thm}

\begin{proof}
(i)\enspace From \eqref{eq:mats}, $G = G(\mathcal O,\id,\rho(\delta))$ 
is generated by
\[
  s = \begin{pmatrix}0&1\\1&0\end{pmatrix},\quad
  g_\varpi = \begin{pmatrix}\varpi&0\\0&\delta\varpi\delta^{-1}\end{pmatrix}
  \enspace\text{and}\enspace
  g_\gamma = \begin{pmatrix}\gamma&0\\0&\delta\gamma\delta^{-1}\end{pmatrix},
\]
where $\varpi = \tfrac12(-1 + i + j + k)$ and $\gamma = \frac1{\sqrt2}(1 + i)$.
The subgroup $G'$ generated by $s$, $g_\varpi$ and $g_\gamma^2$ is
$G(\mathcal T,\id,\rho(\delta))$.  The matrix 
$M = \frac12\begin{pmatrix}-1+k&-\sqrt2 k\\
  \phantom{-}1+k&\phantom{-}\sqrt2\end{pmatrix}$
is unitary and the conjugates $MsM^{-1}$, $Mg_\varpi M^{-1}$ and
$Mg_\gamma M^{-1}$ are of complex type:
\[
  -\tfrac{1}{\sqrt2}\begin{pmatrix}1&1\\1&-1\end{pmatrix},\quad
  \tfrac12\begin{pmatrix}-1+i&-1-i\\\phantom{-}1-i&-1-i\end{pmatrix}
  \enspace\text{and}\enspace
  \tfrac{1}{\sqrt2}\begin{pmatrix}1&-i\\-i&1\end{pmatrix}.
\]
All the reflections in these groups have order 2: there are 12 in
$G'$ and 18 in $G$ (see Lemma \ref{lemma:gen}).  It 
follows from Table D.1 of \cite{lehrer-taylor:2009} that 
$MG'M^{-1}$ is $\ST(12)$ and $MGM^{-1}$ is $\ST(13)$.

The generators of $G'' = G(\mathcal I,\id,\rho(j))$ are
\[
  s = \begin{pmatrix}0&1\\1&0\end{pmatrix},\quad
  g_i = \begin{pmatrix}i&0\\0&-i\end{pmatrix}
  \enspace\text{and}\enspace
  g_\sigma = \begin{pmatrix}\sigma&0\\0&j\sigma j^{-1}\end{pmatrix},
\]
where $\sigma = \tfrac12(\tau^{-1} + i + \tau j)$. The conjugates
of these matrices by $S = R_{1,k}$ are
\[
  \begin{pmatrix}0&-1\\-1&0\end{pmatrix},\quad
  \begin{pmatrix}i&0\\0&-i\end{pmatrix}
  \enspace\text{and}\enspace
  \tfrac12\begin{pmatrix}\tau^{-1}+i&\tau i\\\tau i&\tau^{-1} - i\end{pmatrix}.
\]
Conjugating again by $D = \begin{pmatrix}1&0\\0&-1\end{pmatrix}$ and noting
that there are 30 reflections in $G''$ makes it clear from 
\cite{lehrer-taylor:2009} that $MG''M^{-1}$ is $\ST(22)$, where $M = DS$.

\medskip
\noindent
(ii)\enspace
If $\theta^2 = -1$ and $r\in\R$, the reflection $R_{r,\theta}$ 
conjugates $s$ to $-s$ and commutes with 
$\begin{pmatrix}\xi&0\\0&\theta\xi\theta^{-1}\end{pmatrix}$ for
all $\xi\in\sS^3$. Taking $\theta = \delta$ and $\theta = j$ completes
the proof.
\end{proof}

\begin{rem}\label{rem:inf2}
If $\soi{u,v}$ is a system of imprimitivity for $G$, $G'$ or 
$G''$, then $\soi{uM^{-1},vM^{-1}}$ is a system of imprimitivity for
their conjugates.  Thus the quaternionic reflection groups $\ST(12)$,
$\ST(13)$ and $\ST(22)$ of complex type have infinitely many systems of 
imprimitivity even though, as complex reflection groups, they are primitive.
(These are the only primitive complex reflection groups of rank two all of
whose reflections have order 2.)
\end{rem}

\subsection{Conjugacy of groups of the binary dihedral family}\label{sec:bconj}
\begin{thm}\label{thm:ST}
For all $m \ge 1$, 
$G(\mathcal D_m,\id,\psi_1)$ is conjugate in $\U_2(\HH)$ to
$G(\mathcal C_{2m},\mathcal C_2,\varphi)\simeq\ST(2m,m,2)$, where 
$\varphi$ is inversion. For all $r\in\R$, the reflections $R_{r,i}$ 
normalise $G(\mathcal D_m,\id,\psi_1)$.
\end{thm}

\begin{proof}
The matrices 
$\begin{pmatrix}0&1\\1&0\end{pmatrix}$,
$\begin{pmatrix}\zeta_{2m}&0\\0&\zeta_{2m}\end{pmatrix}$
and $\begin{pmatrix}j&0\\0&-j\end{pmatrix}$ generate 
$G(\mathcal D_m,\id,\psi_1)$.
Their conjugates by $\begin{pmatrix}1&0\\0&j\end{pmatrix}T$ where $T$ 
is defined in \eqref{eq:RST} are
\[
  \begin{pmatrix}1&\phantom{-}0\\0&-1\end{pmatrix},\quad
  \begin{pmatrix}\zeta_{2m}&0\\0&\zeta_{2m}^{-1}\end{pmatrix}\enspace
  \text{and}\enspace
  \begin{pmatrix}\phantom{-}0&1\\-1&0\end{pmatrix}.
\]
They generate $G(\mathcal C_{2m},\mathcal C_2,\varphi)$, which by
Theorem \ref{thm:cyclic} is conjugate to $\ST(2m,m,2)$.

Since $\psi_1 = \rho(i)$, the proof of Theorem \ref{thm:primST} (ii) shows
that for $r\in\R$, $R_{r,i}$ normalises $G(\mathcal D_m,\id,\psi_1)$.
\end{proof}

\begin{rem}\label{rem:inf1}
The centraliser of $G = G(\mathcal D_m,\id,\psi_1)$ in the 
algebra of $2\times 2$ quaternionic matrices is the real subalgebra 
$\left\{\begin{pmatrix}p&qi\\qi&p\end{pmatrix}\;\Big|\; p,q\in\R\right\}$.
The rows of these matrices are systems of imprimitivity for $G$.  
Therefore the groups $\ST(2m,m,2)$ of complex type have infinitely many 
systems of imprimitivity.
\end{rem}

\begin{thm}\label{thm:DmC2}
For all odd $m\ge 1$, $G(\mathcal D_m,\mathcal C_2,\psi_1)$ 
is conjugate to $G(\mathcal D_{2m},\id,\psi_{2m-1})$ in~$\U_2(\HH)$.
\end{thm}

\begin{proof}
The matrices
\[
  s = \begin{pmatrix}0&1\\1&0\end{pmatrix},\quad
  f = \begin{pmatrix}-1&0\\0&1\end{pmatrix},\quad
  g = \begin{pmatrix}\zeta_{2m}&0\\0&\zeta_{2m}\end{pmatrix}\quad\text{and}\quad
  h = \begin{pmatrix}j&0\\0&-j\end{pmatrix}.
\]
generate $G = G(\mathcal D_m,\mathcal C_2,\psi_1)$. The conjugates of these 
matrices by $R_{1,i}$ are
\[
  s' = \begin{pmatrix}0&-1\\-1&0\end{pmatrix},\quad
  f' = \begin{pmatrix}0&-i\\i&0\end{pmatrix},\quad
  g' = \begin{pmatrix}\zeta_{2m}&0\\0&\zeta_{2m}\end{pmatrix}\quad\text{and}\quad
  h' = \begin{pmatrix}j&0\\0&-j\end{pmatrix}.
\]
Thus $s'h'^2 = \begin{pmatrix}0&1\\1&0\end{pmatrix}$ and the order of $s'f'g'$ 
is $4m$. Therefore $\(s'h'^2, s'f'g', h'\) = G(\mathcal D_{2m},\id,\varphi)$,
where $\varphi(i\zeta_{2m}) = -i\zeta_{2m}$ and $\varphi(j) = -j$. Now the
matrix $\begin{pmatrix}1&0\\0&j\end{pmatrix}$ conjugates the group
to $G(\mathcal D_{2m},\id,\psi_{2m-1})$, as required.
\end{proof}

\begin{rem}
If $m$ is even, the reflections $r_0 = T$ and $r_1 = R_{1,i}$ of
\eqref{eq:RST} normalise $G = G(\mathcal D_m,\mathcal C_2,\psi_1)$. 
If $m\equiv 0\pmod 4$, $\(G,r_0\)$ and $\(G,r_1\)$ are isomorphic to 
$G(\mathcal D_m,\mathcal C_4,\psi_1)$; if $m\equiv 2\pmod 4$, they are 
isomorphic to $G(\mathcal D_{2m},\mathcal C_2,\psi_{m-1})$. For all even
$m$, $W = \(G,r_0,r_1\)$ is isomorphic to $\mathcal C_{2m}\boxdot\mathcal O$, 
which (for $m > 2$) is a primitive group whose complexification is imprimitive. 

However, if $m = 2$, the systems of imprimitivity for 
$E = G(\mathcal D_2,\mathcal C_2,\psi_1)$ are the rows of $r_0$, $r_1$, 
$r_2 = R_{1,j}$ and $r_3 = R_{1,k}$. The rows of $r_2$ and $r_3$
remain systems of imprimitivity for $W = \(E,r_0,r_1\)$, which is
isomorphic to $\mathcal C_4\boxdot\mathcal O$. In particular,
$\mathcal C_4\boxdot\mathcal O$ is imprimitive. The group
$\(E,r_0,r_1,r_2\)$ is isomorphic to $G(\mathcal O,\mathcal D_2,\id)$
and $\(E,r_0,r_1,r_2,r_3\)$ is --- in the notation of 
\cite[Table III]{cohen:1980} --- the primitive group $W(P_3)$ with
primitive complexification. The subgroup 
$S = \(r_0,r_1,r_2,r_3\)\simeq G(\mathcal I,\id,\Theta)$ is a double cover 
of $\Sym(5)$, $W(P_3) = ES$, $\mathcal D_2\circ\D_4\simeq E\normal W(P_3)$ 
and $E\cap S = Z(E)$.
\end{rem}

\subsection{Extended binary polyhedral groups}
From the description of $\mathcal O$ in Theorem 5.14 of 
\cite{lehrer-taylor:2009}, the group $\mathcal C_4\boxdot\mathcal O$
of \S\ref{sec:extST} is generated by
\[
  g_0 = \begin{pmatrix}i&0\\0&i\end{pmatrix},\enspace
  g_1 = \tfrac12\begin{pmatrix}-1+i&\phantom{-}1+i\\-1+i&-1-i\end{pmatrix},
  \enspace
  g_2 = \tfrac{1}{\sqrt2}\begin{pmatrix}1+i&0\\0&1-i\end{pmatrix}\enspace
  \text{and}\enspace
  s = \begin{pmatrix}0&j\\-j&0\end{pmatrix}.
\]
The subgroup $\(g_1, g_0g_2, s\)$ is $\mathcal C_4\boxdot_2\mathcal O$. 
Similarly $\mathcal C_4\boxdot\mathcal I$ is generated by
\[
  \begin{pmatrix}i&0\\0&i\end{pmatrix},\quad
  \tfrac12\begin{pmatrix}\tau^{-1}-\tau i&1\\-1&\tau^{-1}+\tau i\end{pmatrix},
  \quad
  \begin{pmatrix}i&0\\0&-i\end{pmatrix}\enspace
  \text{and}\enspace
  \begin{pmatrix}0&j\\-j&0\end{pmatrix}.
\]

Direct calculation shows that the rows of $R_{1,j}$ and $R_{1,k}$ are 
systems of imprimitivity for these groups.

\begin{thm}\label{thm:boxdot}
The groups $\mathcal C_4\boxdot\mathcal O$, $\mathcal C_4\boxdot_2\mathcal O$
and $\mathcal C_4\boxdot\mathcal I$ are conjugate in $\U_2(\HH)$ to 
$G(\mathcal O,\mathcal C_2,\id)$, $G(\mathcal O,\id,\beta)$ and 
$G(\mathcal I,\mathcal C_2,\id)$.
\end{thm}

\begin{proof}
The reflection $R_{1,j}$ conjugates 
$\mathcal C_4\boxdot\mathcal O$ to $G = G(K,H,\varphi)$ where 
$K\simeq\mathcal O$ and $H\simeq\mathcal C_2$. It follows from Lemma 
\ref{lemma:1} that $G$ is conjugate to $G(\mathcal O,\mathcal C_2,\id)$.
The same argument shows that $\mathcal C_4\boxdot_2\mathcal O$ is
conjugate to $G(\mathcal T,\mathcal C_2,\rho(\delta)))$. Then by
Theorem~\ref{thm:TO} it is conjugate to $G(\mathcal O,\id,\beta)$.

Similarly $R_{1,j}$ conjugates $\mathcal C_4\boxdot\mathcal I$ to 
$G = G(K,H,\varphi)$ where $K\simeq\mathcal I$ and $H\simeq\mathcal C_2$.
Then $\begin{pmatrix}1&0\\0&k\end{pmatrix}$ conjugates $G$ to 
$G(\mathcal I,\mathcal C_2,\id)$.
\end{proof}

\begin{table}[ht]
\caption{The proper irreducible imprimitive quaternionic reflection groups
$G(K,H,\varphi)$}\label{tbl:5}
\smallskip
\centering
\begin{tabular}{l|l|r}
&\qquad conditions, isomorphisms&order\\
\toprule
$G(\mathcal D_m,\id,\psi_r)$&$r\ne 1$&$8m$\\
$G(\mathcal D_m,\mathcal C_\ell,\psi_r)$&$\ell\ne1$ and $\ell\mid 2m$
  and $m\ell > 2$&$8m\ell$\\
  &$G(\mathcal D_m,\mathcal C_2,\psi_1)\simeq
  G(\mathcal D_{2m},\id,\psi_{2m-1})$, $m$ odd&\\
\midrule
$G(\mathcal D_m,\mathcal D_m,\id)$&&$32m^2$\\
$G(\mathcal D_{2m},\mathcal D_m,\id)$&&$64m^2$\\
\midrule
$G(\mathcal T,\mathcal T,\id)$&&1\,152\\
$G(\mathcal T,\mathcal D_2,\rho(\delta))$&&384\\
$G(\mathcal T,\mathcal C_2,\rho(\delta))$&
   ${}\simeq G(\mathcal O,\id,\beta)\simeq\mathcal C_4\boxdot_2\mathcal O$
   &96\\
\midrule
$G(\mathcal O,\mathcal O,\id)$&&4\,608\\
$G(\mathcal O,\mathcal T,\id)$&&2\,304\\
$G(\mathcal O,\mathcal D_2,\id)$&&768\\
$G(\mathcal O,\mathcal C_2,\id)$&
  ${}\simeq\mathcal C_4\boxdot\mathcal O$&192\\
\midrule
$G(\mathcal I,\mathcal I,\id)$&&28\,800\\
$G(\mathcal I,\mathcal C_2,\id)$&
  ${}\simeq\mathcal C_4\boxdot\mathcal I$&480\\
$G(\mathcal I,\mathcal C_2,\Theta)$&&480\\
$G(\mathcal I,\id,\Theta)$&&240\\
\bottomrule
\end{tabular}

\medskip
$r = 1$ or $1 < r \le n/2$ and $\gcd(r,n) = \gcd(\kappa,\nu) = 1$\\
$\kappa = n/\gcd(n,r+1)$, $\nu = n/\gcd(n,r-1)$, where $n = 2m/\ell$

\smallskip
$\delta = \frac{1}{\sqrt2}(i-j)$,\enspace $\beta\in\Aut(\mathcal O)$ is central,
\enspace $\Theta\in\Aut(\mathcal I)\setminus\Inn(\mathcal I)$
\end{table}

\section*{Acknowledgements}
The proofs in this paper do not depend on computer calculations. However,
the computer algebra system \textsf{Magma} \cite{magma:1997} was an invaluable
aid exploring and validating small examples. \textsf{Magma} code to
construct irreducible quaternionic reflection groups is available at 
\url{https://www.maths.usyd.edu.au/u/don/software.html}.

In 2024 while investigating a question from Shayne Waldron concerning \textsf{Magma}
calculations with quaternionic reflection groups I noticed that there are examples
of such groups that possess more than one system of imprimitivity. This implies that 
there are errors in the original classification. I asked Ulrich Thiel if this was known
and I am grateful to Ulrich and to Johannes Schmitt for providing the reference to the
paper by Yamagishi where further examples were pointed out.  I thank Gwyn Bellamy,
Johannes Schmitt, Ulrich Thiel and Shayne Waldron for ongoing correspondence about 
complex and quaternionic reflection groups.

Thanks also to the referees for helpful suggestions and alerting me to several
infelicities.
\bibliography{ImprimitiveQRG.bib}
\bibliographystyle{srtnumbered}
\end{document}